\documentclass[12pt]{amsart}
\usepackage{graphicx}
\graphicspath{ {./images/} }
\usepackage{float}
\usepackage{amsmath}
\usepackage{amssymb}
\usepackage{enumerate}
\usepackage{comment}
\usepackage{hyperref}
\newtheorem{thm}{Theorem}[section]

\newtheorem{cor}[thm]{Corollary}
\newtheorem{lem}[thm]{Lemma}
\newtheorem{prop}[thm]{Proposition}
\newtheorem*{prop*}{Proposition}
\newtheorem*{prob*}{Problem}
\newtheorem*{thm*}{Theorem}
\newtheorem*{quest*}{Question}

\theoremstyle{definition}
\newtheorem{defn}[thm]{Definition}
\newtheorem{example}[thm]{Example}

\newtheorem*{defn*}{Definition}
\newtheorem{rem}[thm]{Remark}

\newtheorem{rem*}[thm]{Remark}
\numberwithin{equation}{section}

\newcommand{\Z}{\mathbb Z}

\DeclareMathOperator{\E}{\mathbb{E}}

\newcommand{\mG}{\mathcal{G}}
\newcommand{\mH}{\mathcal{H}}
\newcommand{\mL}{\mathcal{L}}
\newcommand{\Spec}{\text{Spec}}

\newcommand{\SML}{\text{SML}}

\newcommand{\norm}[1]{\left\lVert#1\right\rVert}
\title[Ergodic averages and Khintchine recurrence]{Multiple ergodic averages in abelian groups and Khintchine type recurrence}
\date{\today}
\author{Or Shalom}
\thanks{The author is supported by ERC grant ErgComNum 682150.}
\begin{document}
	\begin{abstract}
		Let $G$ be a countable abelian group. We study ergodic averages associated with configurations of the form $\{ag,bg,(a+b)g\}$ for some $a,b\in\mathbb{Z}$. Under some assumptions on $G$, we prove that the universal characteristic factor for these averages is a factor (Definition \ref{nsext}) of a $2$-step nilpotent homogeneous space (Theorem \ref{mainresult}). As an application we derive a Khintchine type recurrence result (Theorem \ref{Khintchine}). In particular, we prove that for every countable abelian group $G$, if $a,b\in\mathbb{Z}$ are such that $aG,bG,(b-a)G$ and $(a+b)G$ are of finite index in $G$, then for every $E\subset G$ and $\varepsilon>0$ the set $$\{g\in G : d(E\cap E-ag\cap E-bg\cap E-(a+b)g)\geq d(E)^4-\varepsilon\}$$ is syndetic. This generalizes previous results for $G=\mathbb{Z}$, $G=\mathbb{F}_p^\omega$ and $G=\bigoplus_{p\in P}\mathbb{F}_p$ by Bergelson Host and Kra \cite{BHK}, Bergelson Tao and Ziegler \cite{BTZ} and the author \cite{OS}, respectively.
	\end{abstract}
	\maketitle
	\section{Introduction}
	Multiple ergodic averages play an important role in ergodic Ramsey theory. In the case of $\mathbb{Z}$-actions they were used by Furstenberg \cite{F1} to prove Szemer\'edi's theorem \cite{Sz} about the existence of arbitrary large arithmetic progressions in sets of positive upper Banach density.
	The goal of this paper is to study the convergence and limit of some multiple ergodic averages associated with $4$-term arithmetic progressions and more general configurations in countable abelian groups. As usual, a $G$-system $X=(X,\mathcal{B},\mu,T_g)$ is a probability space $(X,\mathcal{B},\mu)$ which is regular\footnote{meaning that $X$ is a compact metric
		space, $\mathcal{B}$ is the completion of the $\sigma$-algebra of Borel sets, and $\mu$ is a Borel measure.}, together with an action of a countable abelian group $G$ on $X$ by measure preserving transformations $T_g:X\rightarrow X$. Fix $a,b\in \mathbb{Z}$, a F{\o}lner sequence $\Phi_N$ of $G$ and bounded functions $f_1,f_2,f_3\in L^\infty(X)$, we study the multiple ergodic averages
	\begin{equation} \label{average3}\mathbb{E}_{g\in \Phi_N} f_1(T_{ag}x) f_2(T_{bg}x) f_3(T_{(a+b)g}x)
	\end{equation}
	where $\mathbb{E}_{g\in \Phi_N} = \frac{1}{|\Phi_N|}\sum_{g\in\Phi_N}$.
	The $L^2$-convergence of these averages as $N$ goes to infinity is already known for all countable nilpotent groups (see Walsh \cite{Walsh}). In the case of $\mathbb{Z}$-actions, these averages were studied by Conze and Lesigne \cite{CL84}, \cite{CL87}, \cite{CL88} and by Furstenberg and Weiss \cite{F&W} using the theory of characteristic factors (see Definition \ref{charfactor}). This theory were developed further by Host and Kra \cite{HK} and Ziegler \cite{Z} in order to deduce the convergence of some multiple ergodic averages associated with $\mathbb{Z}$-actions and by Bergelson Tao and Ziegler with $\mathbb{F}_p^\omega$ actions \cite{Berg& tao & ziegler}.\\ This paper is focused on one of the many applications for these structure theorems associated with the Khintchine type recurrence. For example, we begin with the following result by Bergelson Host and Kra \cite{BHK}.
	\begin{thm} \label{Zrecurrence}
		Let $(X,\mathcal{B},\mu,T)$ be an invertible ergodic system. Then, for any measurable set $A\in\mathcal{B}$ and $\varepsilon>0$ the set
		$$\{n\in\mathbb{Z} : \mu(A\cap T^{-n}A\cap T^{-2n} A \cap T^{-3n}A)>\mu(A)^4-\varepsilon\}$$ 
		is syndetic\footnote{Recall that a set $A$ in a group $G$ is syndetic if there exists a finite set $C\subseteq G$ such that $A+C=G$.}.
	\end{thm}
	In \cite{BTZ} Bergelson Tao and Ziegler proved a counterpart to this result for $\mathbb{F}_p^\omega$-systems. This was generalized further by the author in \cite{OS}.
	\begin{thm} \label{Precurrence}
		Let $P$ be a countable multiset of primes with $3<\min_{p\in P} p$ and let $G=\bigoplus_{p\in P}\mathbb{F}_p$. Then for every ergodic $G$-system $(X,\mathcal{B},\mu,\{T_g\}_{g\in G})$, measurable set $A\subseteq X$ and $\varepsilon>0$ the set
		$$\{g\in G : \mu(A\cap T_gA \cap T_{2g} A \cap T_{3g}A)>\mu(A)^4-\varepsilon\} $$
		is syndetic.
	\end{thm}
	In this paper we generalize the above to all countable abelian groups, under the following conditions.
	\begin{thm} [Khintchine type recurrence result for countable abelian groups]
		\label{Khintchine}
		Let $G$ be a countable abelian group and fix $a,b\in\mathbb{Z}$. If $aG$,  $bG$, $(a-b)G$ and $(a+b)G$ are of finite index in $G$, then for every ergodic $G$-system $(X,\mathcal{B},\mu,\{T_g\}_{g\in G})$, measurable set $A\in \mathcal{B}$ and $\varepsilon>0$, the set $$\{g\in G : \mu(A\cap T_{ag}A\cap T_{bg}A\cap T_{(a+b)g}A)\geq \mu(A)^4-\varepsilon\}$$ is syndetic.
	\end{thm}
	In a recent paper, Bergelson and Ferr\'e Mortagues \cite[Theorem 2.8]{BFe} proved an ergodic version of the Furstenberg correspondence principle. A direct application of the above is the following density result.
	\begin{thm} [Density result]
		Let $G$ be a countable abelian group and $a,b\in\mathbb{Z}$ such that $aG$, $bG$, $(b-a)G$ and $(a+b)G$ are of finite index. Let $\Phi_N$ be any F{\o}lner sequence for $G$ and $d_{\Phi}$ be the corresponding upper density. i.e. $d_{\Phi}(E) = \limsup_{N\rightarrow\infty} \frac{|E\cap \Phi_N|}{|\Phi_N|}$. Then for any set $E\subseteq G$ and $\varepsilon>0$, the set
		$$\{g\in G : d_{\Phi}(E\cap E-ag\cap E-bg\cap E-(a+b)g)\geq d_{\Phi}(E)^4-\varepsilon\}$$ is syndetic.
	\end{thm}
	\begin{rem}
		The case of double recurrence. Namely that $$\{g\in G : \mu(A\cap T_{ag}A\cap T_{bg}A)\geq \mu(A)^3-\varepsilon\}$$ is syndetic, is not covered in this paper. This and a more general version of double recurrence can be found in a recent paper by Ackelsberg Bergelson and Best \cite{ABB}.
	\end{rem}
	Roughly speaking, we say that a factor of an ergodic system $X$ is characteristic for an ergodic average if the limit behavior of the average can be reduced to this factor. The assumption on the indices of $aG$, $bG$, $(b-a)G$ and $(a+b)G$ in Theorem \ref{Khintchine} is necessary to ensure that the systems we study in this paper are characteristic for average (\ref{average3}). It is an interesting question under which conditions on $a$ and $b$ the finite index assumptions in Theorem \ref{Khintchine} can be removed. 
	\begin{defn} [Characteristic factors] \label{charfactor}
		Let $G$ be a countable abelian group and let $X$ be an ergodic $G$-system. For $k\in\mathbb{N}$ and $0\not=a_1,...,a_k\in\mathbb{Z}$, we say that a factor $Y$ is characteristic for the tuple $(a_1g,a_2g,...,a_kg)$ if for every bounded functions $f_1,...,f_k\in L^\infty(X)$ and every F{\o}lner sequence $\Phi_N$ of $G$ we have that
		$$\lim_{N\rightarrow\infty}\left(\mathbb{E}_{g\in \Phi_N} \prod_{i=1}^k T_{a_ig} f_i -  \mathbb{E}_{g\in \Phi_N} \prod_{i=1}^k T_{a_ig} E(f_i|Y)\right) = 0$$
		in $L^2$, where $E(f_i|Y)$ denotes the conditional expectation with respect to the factor $Y$.
	\end{defn}
	\begin{rem} \text{ }
		\begin{itemize}
			
			\item {$X$ is a characteristic factor for any tuple.}
			\item{The mean ergodic theorem states that the trivial factor is characteristic for $(g)$.}
			\item{It is well known that for any countable abelian group $G$, the Kronecker factor (the maximal group rotation factor) is a characteristic factor for $(g,2g)$.}
			\item{If $G=\mathbb{Z}$, then the Kronecker factor is also characteristic for $(ag,bg)$ for any $0\not= a,b\in\mathbb{Z}$, and in \cite{F&W} Furstenberg and Weiss proved that the Conze-Lesigne factor is characteristic for $(ag,bg,(a+b)g)$. We will discuss this below.}
		\end{itemize}
	\end{rem}
	In the case of $\mathbb{Z}$-actions, Host and Kra \cite{HK} proved that characteristic factors for the tuple $(g,2g,3g,...,kg)$ are closely related to an infinite version of the Gowers norms.
	\begin{defn}
		[Gowers Host Kra seminorms]
		Let $G$ be a countable abelian group, let $X=(X,\mathcal{B},\mu,\{T_g\}_{g\in G})$ be a $G$-system, let $\phi\in L^\infty (X)$, and let $k\geq 1$ be an integer. The Gowers-Host-Kra seminorm $\|\phi\|_{U^k}$ of order $k$ of $\phi$ is defined recursively by the formula
		\[
		\|\phi\|_{U^1}:=\lim_{N\rightarrow\infty}\frac{1}{|\Phi_N^1|}\|\sum_{g\in\Phi_N^1}\phi\circ T_g\|_{L^2}
		\]
		for $k=1$, and
		\[
		\|\phi\|_{U^k}:=\lim_{N\rightarrow\infty}\left(\frac{1}{|\Phi_N^k|}\sum_{g\in\Phi_N^k}\|\Delta_g\phi\|_{U^{k-1}}^{2^{k-1}}\right)^{1/2^k}
		\]
		for $k\geq 1$, where $\Delta_g\phi(x)=\phi(T_gx)\cdot\overline{\phi}(x)$ and $\Phi_N^1,...,\Phi_N^k$ are arbitrary F{\o}lner sequences\footnote{All of the limits exist and are independent on the choice of the F{\o}lner sequences, see \cite[Lemma A.18]{Berg& tao & ziegler}.}.
	\end{defn}
	These seminorms were first introduced in the spacial case where $G=\mathbb{Z}/N\mathbb{Z}$ by Gowers in \cite{G} where he derived quantitative bounds for Szemer\'edi's theorem \cite{Sz}.\\
	The Host-Kra factors are defined by the following proposition.
	\begin{prop}\label{UCF} Let $G$ be a countable abelian group, let $X$ be an ergodic $G$-system, and let $k\geq 1$. Then, there exists a factor $Z_{<k}(X)$ of $X$ with the property that for every $f\in L^\infty (X)$, $\|f\|_{U^{k}(X)}=0$ if and only if  $E(f|Z_{<k}(X))=0$.
	\end{prop}
	Proposition \ref{UCF} is proved in \cite[Lemma 4.3]{HK} for $G=\mathbb{Z}$ (see  \cite[Lemma A.32]{Berg& tao & ziegler} for general countable abelian groups). In the case of $\mathbb{Z}$-actions, Leibman \cite{L} showed that the $k$-th Host-Kra factor coincides with the $k$-th Ziegler factor \cite{Z}.\ The latter is the universal (minimal) characteristic factor for all the tuples $(a_1g,a_2g,...,a_kg)$, where $0\not=a_1,...,a_k\in\mathbb{Z}$ are distinct. Leibman's proof can be generalized to arbitrary countable abelian groups, assuming that the following subgroups: $(a_i G)_{i=1}^k$, $((a_i-a_j)G)_{1\leq i\not = j\leq k}$ are of finite index in $G$. Otherwise, $Z_{<k}(X)$ may not be a characteristic factor for the tuple $(a_1g,a_2g,...,a_kg)$.\\
	
	Proposition \ref{UCF} leads to the following definition.
	\begin{defn}
		Let $k\geq 1$ be an integer. Let $G$ be a countable abelian group and $X$ be an ergodic $G$-system. We say that $X$ is a system of order $<k$ if it is isomorphic as a $G$-system to its factor $Z_{<k}(X)$.
	\end{defn}
	
	\begin{rem}
		\text{ }
		\begin{itemize}
			\item {The trivial system is the only system of order $<1$.}
			\item{Any ergodic group rotation is of order $<2$ (see \cite{FurRoth}). The converse is also true (see \cite{HK}).}
			\item{Anti-example: No non-trivial weakly mixing system is of finite order.}
		\end{itemize}
	\end{rem}
	\textbf{Convention.} For an ergodic $G$-system $X$, we call $Z_{<2}(X)$ the Kronecker factor and $Z_{<3}(X)$ the C.L.\ factor (named after Conze and Lesigne \cite{CL84}, \cite{CL87}, \cite{CL88}) and we identify $Z_{<2}(X)$ with a group rotation (see Definition \ref{grouprotation:def}). Similarly, if $X=Z_{<2}(X)$ or $X=Z_{<3}(X)$, we say that $X$ is a Kronecker system or a C.L.\ system, respectively.\\
	
	It is well known that the Conze-Lesigne factor is an abelian extension of the Kronecker factor by an abelian group and a C.L.\ cocycle. We define these notions below.
	\begin{defn} [Abelian cohomology]
		Let $G$ be a countable abelian group, let $X$ be a $G$-system and let $(U,\cdot)$ be a compact abelian group. A measurable function $\rho:G\times X\rightarrow U$ is called a cocycle if $\rho(g+g',x)=\rho(g,x)\cdot \rho(g',T_gx)$ for every $g,g'\in G$ and $\mu$-almost every $x\in X$. The abelian extension of $X$ by the cocycle $\rho$ is defined to be the product space
		$$X\times_{\rho} U = (X\times U, \mathcal{B}_X\otimes\mathcal{B}_U,\mu_X\otimes \mu_U,S_g)$$ together with the action $S_g(x,u) = (T_gx,\rho(g,x)u)$. We denote this system by $X\times_\rho U$.
	\end{defn}
	We say that two cocycles $\rho,\rho':G\times X\rightarrow U$ are $(G,X,U)$-cohomologous (or just cohomologous), if there exists a measurable map $F:X\rightarrow U$ such that $\rho(g,x)/\rho'(g,x)=\Delta_g F(x)$ for all $g\in G$ and $\mu$-almost every $x\in X$. It is well known that cohomologous cocycles define isomorphic group extensions\footnote{The isomorphism is given by $(x,u)\mapsto(x,F(x)u)$.}. We let $B(G,X,U)$ denote the group of all coboundaries, these are functions $G\times X\rightarrow U$ of the form $(g,x)\mapsto \Delta_g F(x)$, where $F:X\rightarrow U$ is a measurable map.\\
	
	Observe that the group $U$ acts on the extension $X\times_\rho U$ by measure preserving transformations $V_u(x,v)=(x,uv)$. More generally, given an action of a compact abelian group $A$ on a system $X$ and $f:X\rightarrow U$ is a measurable map, we define $V_af(x)=f(ax)$ and $\Delta_a f(x)= V_af(x)\cdot f(x)^{-1}$.\\
	Below we define the notion of a C.L.\ cocycle with respect to a group $A$. 
	\begin{defn} [Conze-Lesigne cocycles] \label{CLcocycle}
		Let $G$ be a countable abelian group and let $X$ be an ergodic $G$-system. Let $U$ and $A$ be compact abelian groups and suppose that $A$ acts on $X$ by measure preserving transformations. We say that the cocycle $\rho:G\times X\rightarrow U$ is a C.L.\ cocycle with respect to $A$ if for every $a\in A$ there exist a homomorphism $c_a:G\rightarrow U$ and a measurable map $F_a:X\rightarrow U$ such that $$\Delta_a \rho(g,x) = c_a(g)\cdot \Delta_g F(x)$$
		for $\mu$-almost every $x\in X$ and all $g\in G$.
	\end{defn}
	In \cite{F&W} Furstenberg and Weiss proved the following result.\footnote{Furstenberg and Weiss proved this result under the assumption that $X$ is normal. Host and Kra \cite[Lemma 6.2]{HK} gave another proof without this assumption. We also note that the same proof holds for $G$-systems where $G$ is a countable abelian group.}
	\begin{thm} [$Z_{<3}(X)$ is an extension of the Kronecker by a C.L.\ cocycle] \label{CL} 
		Let $(X,\mathcal{B},\mu,T)$ be an ergodic invertible measure preserving system. Then there exist a compact abelian group $U$ and a cocycle $\rho:Z_{<2}(X)\rightarrow U$ such that $Z_{<3}(X)=Z_{<2}(X)\times_\rho U$ and for every $\chi\in\hat U$, $\chi\circ \rho$ is a C.L.\ cocycle with respect to $Z_{<2}(X)$.
	\end{thm}
	\subsection{The Conze-Lesigne factor as a factor of a nilpotent system}
	We briefly and informally explain how the methods we use in the proof of Theorem \ref{Khintchine} differ from the previous cases for $\mathbb{Z}$ and $\bigoplus_{p\in P}\mathbb{F}_p$ (Theorem \ref{Zrecurrence} and Theorem \ref{Precurrence}).\\
	The main difficulty in the proof of these theorems is to show that the Conze-Lesgine factor admits some nilpotent structure. This nilpotent structure leads to a convenient formula for the limit of average (\ref{average3}), which can be used to derive the recurrence result. In the generality of countable abelian groups we only managed to give partial results in this direction. More specifically, we show that for any ergodic system $(X,G)$ there exists an extension $(Y,H)$ (Definition \ref{nsext}) such that the C.L.\ factor of $Y$ has the structure of a $2$-step nilpotent homogeneous space\footnote{We also give a structure theorem for the Conze-Lesgine factor as a double co-set (see Theorem \ref{double coset}), but we do not use this result in the proof of the Khintchine recurrence.}. As usual, we reduce the study of the limit of average (\ref{average3}) to the case where the functions are measurable with respect to the Conze-Lesigne factor $Z_{<3}(X)$ (Theorem \ref{char}). The main difference is that now we have to pull everything up to the extension $Z_{<3}(Y)$. Using the nilpotent structure of $Z_{<3}(Y)$ we derive a formula for the limit of some multiple ergodic averages (Theorem \ref{formula}). This formula is used to deduce the Khintchine type recurrence result in Theorem \ref{Khintchine}.\\
	
	We begin by introducing a notion of an extension outside of the category of $G$-systems. Observe, that for a $G$-system $X=(X,\mathcal{B},\mu,T_g)$ and a countable abelian group $H$ with a surjective homomorphism $\varphi:H\rightarrow G$ there exists a natural $H$-action on $X$ by $S_h = T_{\varphi(h)}$. This leads to the following definition:
	\begin{defn} [Extensions] \label{nsext}
		Let $G$ and $H$ be countable abelian groups. We say that the system $Y=(Y,(S_h)_{h\in H})$ is an extension of $(X,(T_g)_{g\in G})$ if there exists a surjective homomorphism $\varphi:H\rightarrow G$ and a factor map $\pi:Y\rightarrow X$ such that $\pi\circ S_h = T_{\varphi(h)}\circ \pi$ for all $h\in H$.
	\end{defn}
	\begin{example} \label{example}
		Let $G=\Z/2\Z$ act on the space $X=\{-1,1\}$ by $T_gx = x^g$ and let $H=\Z/4\Z$ act on $Y=\{-1,-i,i,1\}$ by $S_h y=y^h$. Then, the system $(Y,H)$ defines an extension of $(X,G)$ with respect to the homomorphism
		\[
		\begin{aligned}
		&\varphi:H\rightarrow G \\ \varphi(h)&=h\mod 2
		\end{aligned}
		\]
		and the factor map $\pi:Y\rightarrow X$,  $\pi(y)=y^2$.
	\end{example}
	In particular, we see from this example that the family of ergodic $H$-extensions can be larger than the family of ergodic $G$-extensions (there is no ergodic $G$-action on $Y$). In example \ref{example2} below we see another advantage of these extensions.\\
	
	The following group was studied by Conze and Lesigne \cite{CL84}, \cite{CL87}, \cite{CL88} and generalized by Host and Kra \cite{HK} for systems of order $<k$, for any $k\in\mathbb{N}$ (see Definition \ref{HKgroup}).
	\begin{defn} [The homogeneous group] \label{CLgroup}
		Let $G$ be a countable abelian group, let $X$ be a C.L.\ $G$-system and write $X=Z_{<2}(X)\times_\rho U$ for some compact abelian group $U$. For every $s\in Z_{<2}(X)$ and $F:Z_{<2}(X)\rightarrow U$, let $S_{s,F}\in Z_{<2}(X)\ltimes \mathcal{M}(Z_{<2}(X),U)$ be the measure preserving transformation $S_{s,F}(z,u) = (sz,F(z)u)$. The C.L.\ group is given by
		$$\mG(X) = \{S_{s,F}\in  Z_{<2}(X)\ltimes \mathcal{M}(Z_{<2}(X),U) : \exists c:G\rightarrow U\text{ such that } \Delta_s \rho = c\cdot \Delta F\}$$ with the natural multiplication $S_{s,f}\circ S_{t,h} = S_{st,hV_tf}$.
	\end{defn}
	Equipped with the topology of convergence in measure $\mG(X)$ is a $2$-step nilpotent locally compact polish group.\\

	Our main result is the following structure theorem.
	\begin{thm} [Structure Theorem] \label{mainresult}
		Let $G$ be a countable abelian group and let $X$ be an ergodic $G$-system. Then, there exist an extension $(Y,H)$ and a $2$-step nilpotent locally compact polish group $\mG$ which acts transitively on $Z_{<3}(Y)$ by measure preserving transformations. Moreover, we can take $\mG=\mG(Z_{<3}(Y))$ as in Definition \ref{CLgroup}.
	\end{thm}
	The moreover part in Theorem \ref{mainresult} plays an important role in the proof of the Khintchine type recurrence (Theorem \ref{Khintchine}). More specifically, it is used in the proof of the limit formula for some multiple ergodic averages (Theorem \ref{formula}), see Remark \ref{property H} for more details.
	\begin{rem}
	    In \cite{Rud}, Rudolph gave an example of an invertible measure preserving system $(X,T)$ of order $<3$ (i.e. $X=Z_{<3}(X)$) which is not isomorphic to a $2$-step nilpotent homogeneous space. In this paper we show that one can avoid such examples by assuming that the group of eigenfunctions of $X$ is divisible (see Theorem \ref{main:thm}). In Theorem \ref{divisible_extension1} we show that every ergodic $G$-system $X$ admits an extension with that property.
	\end{rem}
	The remark below contains important facts about the structure of $Z_{<3}(Y)$ as a homogeneous space. All of the properties in this remark are proved in the proof of Theorem \ref{main:thm}.
	\begin{rem}
		\label{properties} 
		In the settings of Theorem \ref{mainresult}, the system $Z_{<3}(Y)$ is isomorphic to the $G$-system $(\mG(Z_{<3}(Y))/\Gamma,\mathcal{B},\mu, R_g)$ where $\Gamma$ is the stabilizer of some $x_0\in Z_{<3}(Y_0)$, $\mathcal{B}$ is the Borel $\sigma$-algebra and $\mu$ the Haar measure\footnote{This measure exists because locally compact nilpotent groups are uni-modular.}. Moreover, $\Gamma$ is a totally disconnected closed co-compact subgroup of $\mG(Z_{<3}(Y))$ and there exists a homomorphism $\phi:G\rightarrow \mG(Z_{<3}(Y))$ such that the action $R_g$ is given by left multiplication by $\phi(g)$. 
	\end{rem}
	The factor map $\pi:Y\rightarrow X$, induces a factor $\tilde{\pi}: Z_{<3}(Y)\rightarrow Z_{<3}(X)$ and the following diagram commutes.
	
	\begin{figure}[H]
		\includegraphics{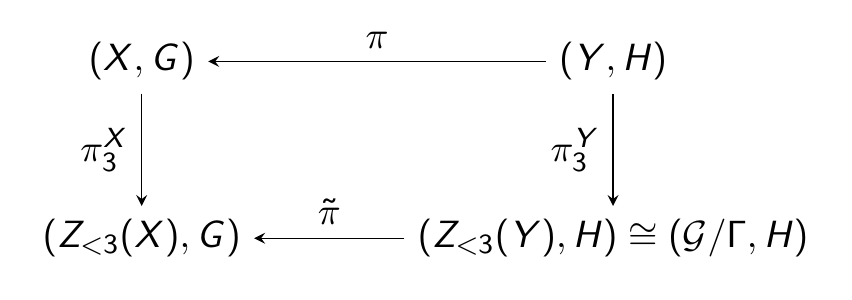}
	\end{figure}
	
	Theorem \ref{mainresult} shows that every ergodic C.L.\ system admits an extension with nilpotent structure. Below we prove a structure theorem for the C.L.\ factor itself (without passing to an extension).\footnote{I thank Yonatan Gutman who informed me that double cosets may be relevant in this work.}
	\begin{thm}[C.L.\ systems are double co-sets] \label{double coset}
	    Let $G$ be a countable abelian group and let $X$ be an ergodic $G$-system. There exists a $2$-step nilpotent locally compact polish group $\mathcal{G}$, a compact totally disconnected subgroup $K\leq \mathcal{G}$ and a closed totally disconnected subgroup $\Gamma\leq \mathcal{G}$ such that $Z_{<3}(X)\cong K\backslash \mathcal{G}/\Gamma$ where $G$ acts on $K\backslash \mathcal{G}/\Gamma$ through a homomorphism $\varphi:G\rightarrow \mathcal{G}$ whose image commutes with $K$.
	\end{thm}
	A system $(X,G)$ is called a $k$-step nilsystem if it is isomorphic to a homogeneous space $\mG/\Gamma$ where $\mG$ is a $k$-step nilpotent Lie group, $\Gamma$ is a discrete co-compact subgroup and there exists a homomorphism $\phi:G\rightarrow\mathcal{G}$ such that $g\in G$ acts on $X$ by a left multiplication by $\phi(g)$. In \cite{CL84}, \cite{CL87}, \cite{CL88} Conze and Lesigne proved that the C.L.\ factor of an ergodic $\mathbb{Z}$-system is isomorphic to an inverse limit of $2$-step nilsystems. Host and Kra \cite{HK} and Ziegler \cite{Z} generalized this result by showing that for every $k\in\mathbb{N}$, a $\mathbb{Z}$-system of order $<k+1$ is an inverse limit of $k$-step nilsystems. Let $(\mG/\Gamma,R_a)$ be a $k$-step $\mathbb{Z}$-nilsystem where $R_a$ is a left translation by some $a\in \mG$. We denote by $\mu_\mG$ the Haar measure on $\mG$. For every $1\leq r\leq k+1$ let $\mG_r$ be the closed subgroup generated by the commutators of length $r$ in $\mG$ and let $m_r$ denote the Haar-measure on the quotient space $\mG_r/\Gamma_r$ where $\Gamma_r = \Gamma\cap \mG_r$. In \cite{Z1} Ziegler proved the following limit formula.
	\begin{thm} \label{Zformula}
		Let $X=(\mG/\Gamma,T)$ be a connected simply connected $k$-step nilsystem and let $f_1,...,f_{k+1}\in L^\infty(X)$. Then for $\mu_\mG$-almost every $x\in \mG$ we have
		\[
		\begin{aligned}
		\lim_{N\rightarrow\infty}\sum_{n=0}^{N-1} \prod_{i=1}^{k+1} T^{in} f_i(x\Gamma) &=\\ 
		\int_{\mG/\Gamma}\int_{\mG_2/\Gamma_2}...\int_{\mG_k/\Gamma_k}& \prod_{i=1}^{k+1} f_i (x\cdot \prod_{j=1}^i y_j^{\binom{i}{j}}\Gamma) \prod_{i=1}^{k+1}dm_i(y_i\Gamma_i). 
		\end{aligned}
		\]
	\end{thm}
    Then, in \cite{BHK} Bergelson Host and Kra generalized this result for non-connected nilsystems which satisfy that $\mG$ is generated by its connected component and $a$.\\
	In \cite{Berg& tao & ziegler} Bergelson Tao and Ziegler studied the structure of the universal characteristic factors associated with $\mathbb{F}_p^\omega$-actions. They showed that any ergodic $\mathbb{F}_p^\omega$-system has the structure of a Weyl system\footnote{A Weyl system is a tower of abelian extensions of the trivial system where the cocycles are phase polynomials (see Definition \ref{phasepoly:def}).} and proved a similar limit formula for multiple ergodic averages associated with this group \cite{BTZ}. Any Weyl system of order $<k+1$ has the structure of a $k$-step nilpotent homogeneous space. This structure theorem was generalized by the author \cite{OS} for $\bigoplus_{p\in P}\mathbb{F}_p$-systems in the special case $k=2$ and in \cite{OS2} for general $k\in\mathbb{N}$.\\
	
	A key component in the proof of Theorem \ref{mainresult} is a result about the (point) spectrum of the $G$ action as a unitary operator on $L^2(X)$. Let $k\geq 1$. We say that a measurable function $P:X\rightarrow S^1$ is a phase polynomial of degree $<k$ if for any $g_1,...,g_{k}\in G$ we have $\Delta_{g_1}...\Delta_{g_k}P=1$ and write $P_{<k}(X,S^1)$ for the group of all phase polynomials of degree $<k$. The $k$-th spectrum of $X$ is defined to be the group
	$$\Spec_k(X)=\{\lambda:G^k\rightarrow S^1 : \exists P\in P_{<k+1}(X,S^1) \text{ s.t. } \forall g_1,...,g_k, \text { } \lambda(g_1,...,g_k)=\Delta_{g_1}...\Delta_{g_k} P \}.$$
	In Theorem \ref{divisible_extension} below we show that for any ergodic system $(X,G)$ there exists an extension $(Y,H)$ such that for every element in $\lambda\in \Spec_k(X)$ and $n\in\mathbb{N}$ there is an $n$-th root for the corresponding element in $\Spec_k(Y)$. In the special case where $k=1$, we show (Theorem \ref{divisible_extension1}) that $P_{<2}(Y,S^1)$ is a divisible group. The following example illustrates this phenomenon in a simple case.
	\begin{example} \label{example2}
		Let $(X,G)$ and $(Y,H)$ be as in Example \ref{example}.\\ The vector space of eigenfunctions of $X$ is spanned by the constant $1$ and the embedding $\chi:X\rightarrow S^1$, $\chi(x)=x$. This finite group $\{1,\chi\}$, under pointwise multiplication, is not divisible. For instance, because there is no square root for $\chi$. On the other hand, let $\chi\circ \pi$ be the lift of $\chi$ to $Y$. We see that the eigenfunction $\tau:Y\rightarrow S^1$, $\tau(x)=x$ is a square root of $\chi\circ\pi$.\\
		This process can be iterated infinitely many times using inverse limits. The result is an extension of $X$ with a divisible group of eigenfunctions into $S^1$. We do this in detail in section \ref{ext:sec} (see also Example \ref{rootexample}).
	\end{example}
	If $(Y,H)$ is a Conze-Lesigne system with a divisible $1$-spectrum, then $\mG(Y)$ acts transitively on $Y$ (Theorem \ref{main:thm}). It is natural to ask whether the same holds for systems of higher order.
	\begin{quest*}
		Let $k\geq 3$, let $G$ be a countable abelian group, and let $(X,G)$ be an ergodic system such that $\Spec_1(X),...,\Spec_{k-1}(X)$ are divisible. Is it true that $Z_{<k+1}(X)$ has the structure of a $k$-step nilpotent homogeneous space? More specifically, is it true that the Host-Kra group $\mathcal{G}(X)$ (Definition \ref{HKgroup}) acts transitively on $X$?
	\end{quest*}
\textbf{Acknowledgment} I would like to thank my adviser Prof. Tamar Ziegler for many valuable discussions and suggestions.

	\section{The Conze-Lesigne factor is characteristic} \label{char:sec}
	In this section we prove that under the assumptions in Theorem \ref{Khintchine}, the C.L.\ factor is a characteristic factor for the tuple $(ag,bg,(a+b)g)$. Our main tool is the van der Corput lemma, (see e.g. \cite{Be}).
	\begin{lem} [van der Corput lemma]
		Let $\mathcal{H}$ be a Hilbert space and $G$ be an amenable group. Then, for every F{\o}lner sequence $\Phi_N$ and any bounded sequence $\{x_g\}_{g\in G}\subseteq \mathcal{H}$ we have:  If $\lim_{N\rightarrow\infty}\mathbb{E}_{g\in\Phi_N}\left<x_{g+h},x_g\right>$ exists for every $h\in G$ and there exists $M\in\mathbb{R}$ such that for any F{\o}lner sequence $\Psi_H$, \begin{equation}\label{cont}\limsup_{H\rightarrow\infty} |\mathbb{E}_{g\in\Psi_H} \lim_{N\rightarrow\infty}\mathbb{E}_{g\in\Phi_N}\left<x_{g+h},x_g\right>|\leq M.
		\end{equation} Then,
		$$\limsup_{N\rightarrow\infty} \|\mathbb{E}_{g\in\Phi_N}x_g\|^2 \leq M.$$
		In particular, if $\lim_{H\rightarrow\infty}\mathbb{E}_{g\in\Psi_H}\lim_{N\rightarrow\infty}\mathbb{E}_{g\in\Phi_N}\left<x_{g+h},x_g\right> =0$, then $\lim_{N\rightarrow\infty}\mathbb{E}_{g\in\Phi_N} x_g=0$.
	\end{lem}
	
	\begin{proof}
		Let $\varepsilon>0$ be arbitrary. By the properties of a F{\o}lner sequence, we have that for sufficiently large $N$ and $H$, 
		$$\|\mathbb{E}_{g\in\Phi_N} x_g - \mathbb{E}_{h\in\Phi_H} \mathbb{E}_{g\in\Phi_N} x_{g+h}\|<\varepsilon.$$
		We use $o(\varepsilon)$ to denote a positive quantity that goes to $0$ as $\varepsilon\rightarrow 0$. Since $x_g$ is bounded, the triangle inequality gives 
		$$\|\mathbb{E}_{g\in\Phi_N} x_g\|^2 \leq \mathbb{E}_{g\in\Phi_N} \|\mathbb{E}_{h\in \Phi_H}x_{g+h}\|^2+o(\varepsilon).$$
		Then, the right hand side becomes
		$$\mathbb{E}_{g\in \Phi_N}  \mathbb{E}_{h\in\Phi_H}\mathbb{E}_{h'\in\Phi_H} \left<x_{g+h},x_{g+h'}\right> + o(\varepsilon).$$
		We make a change of variables and change the order of summation.
		$$ \label{quantity} \mathbb{E}_{h'\in\Phi_H}\mathbb{E}_{h\in\Phi_H} \mathbb{E}_{g\in \Phi_N+h'} \left<x_{g+h-h'},x_{g}\right> + o(\varepsilon).$$
		As $\Phi_N$ is a F{\o}lner sequence, taking a limit as $N\rightarrow\infty$ we get that for sufficiently large $H$, the above equals to
		$$\mathbb{E}_{h'\in\Phi_H} \mathbb{E}_{h\in\Phi_H} \gamma_{h-h'} + o(\varepsilon).$$
		Making a change of variables again this becomes
		\begin{equation}\label{final}\mathbb{E}_{h'\in \Phi_H}\mathbb{E}_{h\in \Phi_H+h'}  \gamma_h + o(\varepsilon).
		\end{equation}
		Let $\varepsilon_1>0$, and suppose by contradiction that there exists a subsequence $H_k\underset{k\rightarrow\infty}{\longrightarrow}\infty$ such that for every $k$, $$\left|\mathbb{E}_{h'\in \Phi_{H_k}} \mathbb{E}_{h\in \Phi_{{H_k}+h'}}\gamma_h\right| > M+\varepsilon_1.$$
		Then we can find $h'_k\in \Phi_{H_K}$ such that $$\left|\mathbb{E}_{h\in \Phi_{{H_k}+h_k'}}\gamma_h\right| > M+\varepsilon_1.$$
		However, $\Psi_k = \Phi_{{H_k}+h'_k}$ is a F{\o}lner sequence and we have a contradiction to (\ref{cont}). Therefore the $\limsup_{H\rightarrow\infty}$ of (\ref{final}) is bounded above by
		$M + o(\varepsilon)$. As $\varepsilon>0$ is arbitrary the claim follows.
	\end{proof}
	
	The first application of this lemma is the following result of Furstenberg and Weiss \cite{F&W}.
	\begin{lem} [The Kronecker factor is characteristic for double averages] \label{Kronecker}
		Let $G$ be a countable abelian group and let $X$ be an ergodic $G$-system. Suppose that $a,b\in\mathbb{Z}$ are such that $aG,bG$ and $(b-a)G$ are of index $d_a,d_b$ and $d_{b-a}$ in $G$, respectively. Fix $f_1,f_2\in L^\infty(X)$ with $\|f_1\|_\infty,\|f_2\|_\infty\leq 1$ and let $x_g = T_{ag}f_1\cdot T_{bg}f_2$. Then $\lim_{N\rightarrow\infty}\mathbb{E}_{g\in\Phi_N} x_g$ exists and
		$$\norm{\lim_{N\rightarrow\infty}\mathbb{E}_{g\in\Phi_N} x_g}_{L^2(X)}^2\leq d_{b-a}\cdot\min\{ d_a\cdot \|f_1\|_{U^2(X)},  d_b\cdot \|f_2\|_{U^2(X)}\}$$
		in $L^2$ for every F{\o}lner sequence $\Phi_N$ of $G$.
	\end{lem}
	\begin{proof}
		We follow the argument in \cite{F&W}. Set $x_g = T_{ag}f_1\cdot T_{bg} f_2$ then,
		$$\left< x_{g+h}, x_{g}\right> = \int_X T_{ag+ah}f_1 \cdot T_{bg+bh} f_2 \cdot T_{ag} \overline{f}_1 \cdot T_{bg} \overline{f}_2 d\mu.$$
		Since $T_{ag}$ is measure preserving we have,
		$$\lim_{N\rightarrow\infty} \mathbb{E}_{g\in \Phi_N} \left< x_{g+h}, x_{g}\right> = \lim_{N\rightarrow\infty} \mathbb{E}_{g\in \Phi_N} \int_X \Delta_{ah} f_1 \cdot T_{(b-a)g} \Delta_{bh} f_2 d\mu.$$
		By the mean ergodic theorem the limit exists and equals to
		\begin{equation} \label{proj}
		\int_X \Delta_{ah} f_1 P_{b-a}(\Delta_{bh} f_2) d\mu
		\end{equation} where $P_{b-a}$ is the projection to the $(b-a)G$-invariant functions. 
		If $(b-a)G$ is ergodic, then this equals to $$\int_X \Delta_{ah} f_1 d\mu \cdot \int_X \Delta_{bh} f_2 d\mu.
		$$
		The limit of the average of this in absolute value
		$$\limsup_{H\rightarrow\infty} \mathbb{E}_{h\in \Phi_N} \left|\int_X \Delta_{ah} f_1 d\mu \cdot \int_X \Delta_{bh} f_2 d\mu\right|$$ is bounded by $\min\{d_a\cdot\|f_1\|_{U^2}\cdot ,d_b\cdot \|f_2\|_{U^2}\}$ and the claim follows by the van der corput lemma. If $(b-a)G$ is not ergodic, then since $(b-a)G$ is of index $d_{b-a}$ in $G$ there are at most $d_{b-a}$ ergodic components. In particular, we can find a partition of $X$ to $(b-a)G$-invariant sets, $X=\bigcup_{i=1}^{d_{b-a}} A_i$ such that $P_{b-a}$ is an integral operator with kernel $\sum_{i=1}^{d_{b-a}}1_{A_i}(x)1_{A_i}(y)$. We conclude that (\ref{proj}) equals to
		$$ \int_X \int_X \overline{f}_1(x)\cdot \overline{f}_2 (y) \cdot T_{ah} f_1(x)\cdot T_{bh} f_2(y) \sum_{i=1}^{d_{b-a}} 1_{A_i}(x) 1_{A_i}(y) d\mu(x) d\mu(y).$$
		Taking another average on $h$ over any F{\o}lner sequence $\Psi_H$ and applying the mean ergodic theorem for the action of $T_{ah}\times T_{bh}$, the limit of the above becomes
		\begin{equation} \label{averageproj} \int_X \int_X \overline{f}_1(x)\cdot \overline{f}_2 (y)  \sum_{i=1}^{d_{b-a}} H(x,y) 1_{A_i}(x) 1_{A_i}(y) d\mu(x) d\mu(y)
		\end{equation}
		for some bounded $T_{ah}\times T_{bh}$-invariant function $H(x,y)$. It is classical that every $T_{ah}\times T_{bh}$-invariant function can be written by sums of all products of $d_a$ eigenfunctions in $x$ and $d_b$ eigenfunctions in $y$. Since $1_{A_i}(x)$ is $T_{(b-a)h}$-invariant, it is a sum of $d_{b-a}$ eigenfunctions. Let $Z$ be the Kronecker factor, we conclude that the term in equation (\ref{averageproj}) is bounded by the minimum between $d_{b-a}\cdot d_a\cdot \max_{\chi\in \hat Z} |\left<f_1,\chi\right>|$ and $d_{b-a}\cdot d_b\cdot \max_{\chi\in \hat Z} |\left<f_2,\chi\right>|$.\\ Since the $U^2$-norm bounds the maximal Fourier coefficient the claim follows. To see this let $f\in L^2(X)$ be any function. We can decompose $f$ with respect to the orthogonal projection $E(\cdot|Z)$ and write $f=\sum_{\chi\in \hat Z} \left<f,\chi\right>\cdot \chi + f'$, then $$\|f\|_{U^2}^4=\|E(f|Z)\|_{U^2}^4 =  \sum_{\chi\in \hat Z} |\left<f,\chi\right>|^4 \geq \max_{\chi\in\hat Z} |\left<f,\chi\right>|^4.$$
		This clearly implies that $\|f\|_{U^2}\geq\max_{\chi\in \hat Z} |\left<f,\chi\right>|$, and therefore, the van der Corput lemma gives the promised inequality.\\
		It is left to show that the limit exists. By linearity we can reduce matters to the Kronecker factor. For $i=1,2$ let $\tilde{f_i}=E(f_i|Z)$. Then, by approximating $\tilde{f}_1,\tilde{f}_2$ by linear combinations of eigenfunctions direct computation gives,
		$$\lim_{N\rightarrow\infty}\mathbb{E}_{g\in\Phi_N} T_{ag} \tilde{f}_1(x)\cdot T_{bg} \tilde{f}_2(x) = \int_Z \tilde{f}_1(xy^a)\tilde{f}_2(xy^b) d\mu_Z(y)$$
		in $L^2$, where here we abuse notation and view $\tilde{f}_1$ and $\tilde{f}_2$ as functions on $Z$. This completes the proof.
	\end{proof}
	Now, we generalize this for the tuple $(ag,bg,(a+b)g)$.
	\begin{prop} [$Z_{<3}(X)$ is characteristic for triple averages] \label{char}
		Let $a, b\in\mathbb{Z}$ and $G$ be as in Theorem \ref{Khintchine} and let $X$ be an ergodic $G$-system. Let $f_1,f_2,f_3\in L^\infty(X)$ and for every $i=1,2,3$ let $\tilde{f}_i=E(f_i|Z_{<3}(X))$. Then, assuming that the following limits exist in $L^2$, we have $$\lim_{N\rightarrow\infty}\mathbb{E}_{g\in\Phi_N} T_{ag}f_1 T_{bg}f_2 T_{(a+b)g}f_3=\lim_{N\rightarrow\infty}\mathbb{E}_{g\in\Phi_N} T_{ag}\tilde{f}_1 T_{bg}\tilde{f}_2 T_{(a+b)g}\tilde{f}_3.$$ 
	\end{prop} 
	
	\begin{proof}
		Let $d_a,d_b,d_{b-a}$ and $d_{a+b}$ denote the indices of $aG,bG,(b-a)G$ and $(a+b)G$ in $G$, respectively and let $f_1,f_2,f_3\in L^\infty(X)$. By linearity it is enough to show that if either $\tilde{f}_1,\tilde{f}_2$ or $\tilde{f}_3$ is zero, then $$\lim_{N\rightarrow\infty}\mathbb{E}_{g\in\Phi_N} T_{ag}f_1 T_{bg}f_2 T_{(a+b)g}f_3=0.$$
		By the symmetry of the equation we can assume without loss of generality that $\tilde{f}_3=0$. Moreover, if we divide each function by a constant we can also assume that $\|f_1\|_{\infty}$, $\|f_2\|_{\infty}$ and $\|f_3\|_{\infty}$ are bounded by $1$. Set $x_g = T_{ag}f_1\cdot T_{bg}f_2\cdot T_{(a+b)g}f_3$, then for every $g,h\in G$ and $N\in\mathbb{N}$ we have,
		$$\mathbb{E}_{g\in\Phi_N}\left<x_{g+h},x_g\right> = \mathbb{E}_{g\in\Phi_N} \int_X T_{ag+ah}f_1 \cdot T_{bg+bh}f_2 
		\cdot T_{(a+b)(g+h)}f_3 \cdot T_{ag}\overline{f}_1 \cdot T_{bg}\overline{f}_2 \cdot T_{(a+b)g}\overline{f}_3d\mu.$$
		Since $T_{ag}$ is measure preserving the above equals to
		$$ \mathbb{E}_{g\in\Phi_N} \int_X \Delta_{ah}f_1 \cdot T_{(b-a)g} \Delta_{bh}f_2 \cdot T_{bg} \Delta_{(a+b)h}f_3 d\mu.$$
		By the previous lemma this average converges in $L^2$. Observe that by the Cauchy-Schwartz inequality and since $\|f_1\|_\infty\leq 1$, the absolute value of the above is smaller or equal to
		$$\norm{\mathbb{E}_{g\in\Phi_N} T_{(b-a)g} \Delta_{bh}f_2 \cdot T_{bg} \Delta_{(a+b)h}f_3}_{L^2}.$$
		By the previous lemma, the limit as $N\rightarrow\infty$ is bounded by the square root of
		$d_a^2\cdot \norm{\Delta_{(a+b)h}f_3}_{U^2(X)}^2$. Since $\|\cdot\|_{U^3(X)}$ is a seminorm, we conclude that for every F{\o}lner sequence $\Psi_H$,
		$$\lim_{H\rightarrow\infty} \mathbb{E}_{h\in \Psi_H} \norm{\Delta_{(a+b)h}f_3}_{U^2(X)}^2 \leq d_{a+b}\cdot \|f\|_{U^3(X)}^4.$$ Therefore,
		$$\left|\lim_{H\rightarrow\infty}\mathbb{E}_{h\in\Psi_H}\lim_{N\rightarrow\infty} \mathbb{E}_{g\in\Phi_N} \left<x_{g+h},x_g\right> \right| \leq d_a\cdot d_{a+b}\cdot \|f_3\|_{U^3}^2 = 0$$
		and by the van der Corput lemma the claim follows.
	\end{proof}
	
	\section{Generalized spectrum} \label{ext:sec}
	Let $G$ be a countable abelian group, $(X,G)$ an ergodic $G$-system and $k\geq 1$. In this section we construct an extension $(Y,H)$ with the property that every phase polynomial $p:X\rightarrow S^1$ of degree $<k$ admits a phase polynomial $n$-th root $q:Y\rightarrow S^1$ such that $q^n=p\circ \pi$ for every $n\in\mathbb{N}$ where $\pi:Y\rightarrow X$ is the factor map.	We begin with some definitions. First, we generalize the definition of a phase polynomial to functions taking values in an arbitrary compact abelian group.
	\begin{defn}[Phase polynomials] \label{phasepoly:def}
		Let $X$ be an ergodic $G$-system, let $k\geq 0$ and let $U$ be a compact abelian group. We say that a function $P:X\rightarrow U$ is a phase polynomial of degree $<k$ if for every $g_1,...,g_{k}\in G$ we have that $\Delta_{g_1}...\Delta_{g_{k}}P=1_U$. We let $P_{<k}(X,U)$ denote the group of phase polynomials $P:X\rightarrow U$ of degree $<k$. 
	\end{defn}
	Bergelson Tao and Ziegler proved that up to constant multiplication, there are at most countably many phase polynomials in $P_{<k}(X,S^1)$. In other words, the quotient $P_{<k}(X,S^1)/P_{<1}(X,S^1)$ is a countable (discrete) group.\footnote{This assertion follows from the lemma below and the fact that $L^2(X)$ is separable.}
	\begin{lem}  [Separation Lemma] \cite[Lemma C.1]{Berg& tao & ziegler}\label{sep:lem} Let $X$ be an ergodic $G$-system, let $k\geq 1$, and let $\phi,\psi\in P_{<k}(X,S^1)$ be such that $\phi/\psi$ is non-constant. Then $\|\phi-\psi\|_{L^2(X)} \geq \sqrt{2}/2^{k-2}$.
	\end{lem}
	We also recall the following proposition from Appendix \ref{poly:appendix}.
	\begin{prop}\label{orderext}
		Let $G$ be a countable abelian group and $k,m\geq 1$. Let $X$ be an ergodic $G$-system of order $<k$ and let $P:X\rightarrow S^1$ a phase polynomial of degree $<m$. Then \begin{itemize}
			\item {$X$ is an abelian extension of $Z_{<k-1}(X)$ by a compact abelian group $U$.}
			\item{For every $u\in U$, $\Delta_u P$ is a phase polynomial of degree $<\max\{0,m-k+1\}$. In particular, $P$ is measurable with respect to $Z_{<m}(X)$.}
			\item{If $p:G\times X\rightarrow U$ is a phase polynomial cocycle of degree $<k$, then $X\times_p U$ is a system of order $<k$.}
		\end{itemize}
	\end{prop}
	\textit{Spectrum}: The (point) spectrum of a $G$-system $X$ is the group of eigenvalues of the $G$ action on $L^2(X)$. We generalize this notion below.
	\begin{defn} [Generalized spectrum] \label{spec:def}
	Let $X$ be an ergodic $G$-system and $1\leq k\in\mathbb{N}$. We define the $k$-th spectrum of $X$ by
		$$\Spec_k(X)=\{\lambda:G^k\rightarrow S^1 : \exists P\in P_{<k+1}(X,S^1) \text{ s.t. } \forall g_1,...,g_k, \text { } \lambda(g_1,...,g_k)=\Delta_{g_1}...\Delta_{g_k} P \}.$$
	\end{defn}
	We are particularly interested in the case where this group is divisible.
	\begin{prop} [Definition and properties of divisible groups]\label{divisible} A group $(H,\cdot)$ is said to be divisible if for every $h\in H$ and $1\leq n\in\mathbb{N}$ there exists $g\in H$ with $g^n=h$. Divisible groups are injective in the category of discrete abelian groups. Namely, if $H\leq G$ are discrete abelian groups and $H$ is divisible, then $G\cong H\oplus G/H$.
	\end{prop}
	Given two abelian groups $H$ and $G$ and an inclusion $\imath:H\hookrightarrow G$, we say that $H$ is divisible in $G$ if for every $n\in\mathbb{N}$ and $h\in H$ there exists $g\in G$ with $\imath(h)=n\cdot g$. This gives rise to the following definition of divisible systems.
	\begin{defn} [Divisible systems] \label{divisiblesystem:def}
		Let $G$ be a countable abelian group, let $X$ be an ergodic $G$-system, and let $k\geq 2$. We say that $X$ is $k$-divisible if $\Spec_1(X),...,\Spec_{k-1}(X)$ are divisible. Similarly, if $(Y,H)$ is an extension of $X$, then $X$ is $k$-divisible in $Y$ if for every $1\leq i \leq k-1$, $\Spec_i(X)$ is divisible in $\Spec_i(Y)$ with respect to the natural inclusion.\footnote{Let $\lambda\in \Spec_i(X)$, then there exists a phase polynomial $P:X\rightarrow S^1$ such that $\Delta_{g_1}...\Delta_{g_i} P = \lambda(g_1,...,g_i)$. The natural inclusion is the map which sends $\lambda$ to the element $(h_1,...,h_i)\mapsto \Delta_{h_1}...\Delta_{h_i}P\circ \pi$ where $\pi:Y\rightarrow X$ is the factor map.}
	\end{defn}
	If $X$ is $k$-divisible, then the group of phase polynomials of degree $<k$ is divisible. In fact, we prove the following stronger result.
	\begin{thm} [$k$-Divisible implies that $P_{<k}(X,S^1)$ is divisible] \label{roots:thm}
		Let $G$ be a countable abelian group and $k\geq 2$. If $X$ is a $k$-divisible and ergodic $G$-system, then for every $d\leq k$ the group $P_{<d}(X,S^1)$ is divisible.
	\end{thm}
	\begin{proof}
		We prove the claim by induction on $d$. For $d=1$, $P_{<1}(X,S^1)\cong S^1$ and the claim follows. Let $2\leq d\leq k$ and suppose that the claim has already been proven for smaller values of $d$. Fix $P\in P_{<d}(X,S^1)$ and a natural number $n\in\mathbb{N}$ and let $\lambda (g_1,...,g_{d-1})= \Delta_{g_1}...\Delta_{g_{d-1}} P$. Then, by assumption there exists $\gamma\in \Spec_{d-1}(X)$ with $\gamma^n = \lambda$. Let $Q\in P_{<d}(X,S^1)$ be such that $\gamma(g_1,...,g_{d-1}) = \Delta_{g_1}...\Delta_{g_{d-1}} Q$. Then $$\Delta_{g_1}...\Delta_{g_{d-1}} Q^n = \Delta_{g_1}...\Delta_{g_{d-1}}P.$$ We see that $P/Q^n$ is a phase polynomial of degree $<d-1$. By induction hypothesis there exists $Q'\in P_{<d-1}(X,S^1)$ with $Q'^n = P/Q^n$ and therefore $P=(QQ')^n$, as required.
	\end{proof}
	The following proposition will play an important role in the proof of Theorem \ref{mainresult}.
	\begin{prop}[Reducing C.L.\ equations to the circle] \label{typecircle:prop}
		Let $k\geq 2$ and let $X$ be a $k$-divisible and ergodic  $G$-system. Let $\rho:G\times X\rightarrow U$ be a cocycle into a compact abelian group $U$ and suppose that for every $\chi\in \hat U$ there exists a phase polynomial $q_\chi:G\times X\rightarrow U$ of degree $<k-1$ and a measurable map $F_\chi:X\rightarrow U$ such that $\chi\circ\rho = q_\chi\cdot \Delta F_\chi$. Then, there exists a phase polynomial $q:G\times X\rightarrow U$ and a measurable map $F:X\rightarrow U$ such that $\rho = q\cdot \Delta F$.
	\end{prop}
	We note that the proposition above fails if the system is not $k$-divisible. We give an example: Let $X=(\mathbb{R}/\mathbb{Z},\alpha)$ be an irrational rotation on the torus and let $\rho:\mathbb{R}/\mathbb{Z}\rightarrow C_2$ be the cocycle $\rho(x) = e\left(-\frac{\alpha}{2}+\frac{\{x+\alpha\}}{2}-\frac{\{x\}}{2}\right)$ where $\{x\}$ is the fractional part of $x$ and $e(y):=e^{2\pi i y}$. Observe, that as a cocycle into $S^1$, $\rho$ is cohomologous to the constant $e(-\frac{\alpha}{2})$, but not as a cocycle into $C_2$. To see that, let assume by contradiction that $\rho = c\cdot \Delta F$ where $c\in C_2$ and $F:X\rightarrow C_2$. Then, $c\cdot e\left({\frac{\alpha}{2}}\right)$ is an eigenvalue for the eigenfunction $e\left(\frac{\{x\}}{2}\right)\cdot \overline{F(x)}$. This is a contradiction, because the eigenvalues of $X$ are $\{e(n \alpha) : n\in\mathbb{Z}\}$.
	
	\begin{proof}[Proof of Proposition \ref{typecircle:prop}]
		Let $\rho:G\times X\rightarrow U$ be as in the proposition and let $\mathcal{K}$ be the group of all pairs $(\chi,F)$ for which the equation in the claim holds. Namely,
		$$\mathcal{K} = \{(\chi,F)\in \hat U \times \mathcal{M}(X,S^1): \exists q\in P_{<k-1}(G,X,S^1) \text{ s.t. } \chi\circ\rho = c\cdot \Delta F\}.$$ $\mathcal{K}$ is a closed subgroup of the abelian group $\hat U\times \mathcal{M}(X,S^1)$. Moreover, it is easy to see that $\ker p\cong P_{<k}(X,S^1)$ and by the assumptions in the proposition, it follows that the projection $p:\mathcal{K}\rightarrow \hat U$ is onto. Therefore, by Theorem \ref{locallycomp}, $\mathcal{K}$ is a locally compact abelian group and we have a short exact sequence 
		\begin{equation} \label{seq1} 1\rightarrow P_{<k}(X,S^1)\rightarrow \mathcal{K}\rightarrow \hat U \rightarrow 1.
		\end{equation}
		By ergodicity $P_{<1}(X,S^1)\cong S^1$. Let $A=P_{<k}(X,S^1)/P_{<1}(X,S^1)$. Then, by quotienting out $P_{<1}(X,S^1)$ in (\ref{seq1}) we conclude that
		\begin{equation} \label{seq2} 1\rightarrow A\rightarrow \mathcal{K}/S^1\rightarrow \hat U \rightarrow 1
		\end{equation}
		is a short exact sequence. Since  $\hat U$ and $A$ are discrete (by Lemma \ref{sep:lem}), we deduce that so is $\mathcal{K}/S^1$. Moreover, by Theorem \ref{roots:thm} the group $A$ is divisible. Therefore, Proposition \ref{divisible} implies that
		$$\mathcal{K}/S^1\cong A\times \hat U.$$ Since the circle $S^1$ is injective in the category of locally compact abelian groups, the above implies that $\mathcal{K}\cong P_{<k}(X,S^1)\times \hat U$. Thus, we can find a Borel cross section (see Definition \ref{cross section}) $\chi\mapsto (\chi,F_\chi)$ such that $\chi\mapsto F_\chi$ is a homomorphism and for every $\chi\in\hat U$, $\chi\circ \rho = q_\chi \cdot F_\chi$ for some phase polynomial $q_\chi:G\times X\rightarrow S^1$ of degree $<k-1$. It follows that $\chi\mapsto q_\chi$ is also a homomorphism and so, by the Pontryagin duality theorem there exists a measurable map $F:X\rightarrow U$ and a phase polynomial of degree $<k-1$, $q:G\times X\rightarrow U$ such that $F_\chi = \chi\circ F$ and $\chi\circ q = q_\chi$. Since the characters separate points, we conclude that $\rho = q\cdot \Delta F$, as required.
	\end{proof}
	Observe that every countable abelian group is a factor of a group with divisible dual (say $\mathbb{Z}^\omega$). Therefore for the sake of the proof of Theorem \ref{mainresult}, it is enough to assume that the group $G$ has a divisible dual (equivalently, that $G$ is torsion free, see Proposition \ref{torsionfree}).\\
		Let $k\geq 1$, then every element $\lambda\in \Spec_k(X)$ is a multilinear map (i.e. a homomorphism in every coordinate) from $G^k$ to $S^1$. More formally we have the following definition.
	\begin{defn}[Multilinear maps]
		Let $G$ be a countable abelian group, let $X$ be a $G$-system, and let $m\geq 1$. We say that $\lambda:G^m\rightarrow S^1$ is a \textit{multilinear} map if for every $1\leq i \leq m$, $g_1,...,g_m\in G$ and $g_i'\in G$ we have $$\lambda(g_1,...,g_i\cdot g_i',...,g_m) = \lambda(g_1,...,g_i,...,g_m)\cdot \lambda(g_1,...,g_i',...,g_m).$$ We denote by $\text{ML}_{m}(G,S^1)$ the group of multilinear maps $G^m\rightarrow S^1$. We say that a multilinear map $\lambda$ is \textit{symmetric} if it is invariant to the permutations of coordinates and let $\SML_{m}(G,S^1)$ denote the group of symmetric multilinear maps.
	\end{defn}
	
	The groups $\text{ML}_m(G,S^1)$ and $\text{SML}_{m}(G,S^1)$ are the Pontryagin dual of the tensor product and symmetric tensor product of $m$ copies of $G$, respectively. 
	\begin{defn}[Tensor products]
	Let $G$ be a countable abelian group. The $m$-tensor product of $G$ is a group $G^{\otimes m}$ satisfying the following universal property: There exists a multilinear map\footnote{It is common to denote the element $\imath(g_1,...,g_m)$ by $g_1\otimes...\otimes g_m$.} $\imath:G^m\rightarrow G^{\otimes m}$ such that for every multilinear map $\lambda\in \text{ML}_m(G,S^1)$ there exists a homomorphism $\varphi_\lambda : G^{\otimes m}\rightarrow S^1$ such that $\lambda=\varphi_\lambda\circ\imath$. Similarly one can define the symmetric tensor product $G^{\otimes_{sym} m}$.
	\end{defn}
	Note that the tensor product and symmetric tensor product always exist and unique up to isomorphism.
    We recall some basic results about topological groups.
    \begin{prop} \cite[Corollary 8.5, page 377]{HM} \label{torsionfree}
        Let $G$ be a countable (discrete) abelian group. Then $\hat G$ is divisible if and only if $G$ is torsion free. 
    \end{prop}
   The following result will play a significant role in our argument.
	\begin{prop} \label{SML}
		Let $G$ be a countable torsion free abelian group. Then for every $m\geq 1$, $\SML_m(G,S^1)$ is a divisible group. 
	\end{prop}
	\begin{proof}
	By Proposition \ref{torsionfree}, it is enough to show that $G^{\otimes_{sym} m}$ is torsion free. We start with the case where $G$ is finitely generated. Since $G$ is torsion free, it is isomorphic to $\mathbb{Z}^d$ for some $d\in\mathbb{N}$. It is easy to see that $G^{\otimes_{sym} m}$ is a free quotient of $G^{\otimes m}\cong \mathbb{Z}^{d^m}$ and the claim follows. Now, let $G$ be an arbitrary countable torsion free abelian group and assume by contradiction that there exists $0\not =g\in G^{\otimes_{sym} m}$ of finite order. It is well known that the image of the map $\imath:G^m\rightarrow G^{\otimes_{sym} m}$ generates the group $G^{\otimes_{sym} m}$. Therefore, there exists $g_1,...,g_k\in G^m$ such that $g=\varphi(g_1)\cdot...\cdot \varphi(g_k)$. The coordinates of $g_1,...,g_k$ generates a finitely generated subgroup $H$ of $G$ and $g\in H^{\otimes_{sym} m}$. The finitely generated case provides a contradiction.
	\end{proof}
	We need the following result by Zimmer \cite[Corollary 3.8]{Zim}.
	\begin{defn} [Image and minimal cocycles]
		Let $G$ be a countable abelian group, let $X$ be a $G$-system, and let $\rho:G\times X\rightarrow U$ be a cocycle into a compact abelian group $U$. The image of $\rho$ is defined to be the closed subgroup $U_\rho\leq U$ generated by $\{\rho(g,x):g\in G,x\in X\}$. We say that $\rho$ is minimal if it is not $(G,X,U)$-cohomologous to a cocycle $\sigma$ with $U_\sigma\lneqq U_\rho$.
	\end{defn}
	\begin{lem} \label{minimal}
		Let $X$ be an ergodic $G$-system and $\rho:G\times X\rightarrow U$ be a cocycle into a compact abelian group $U$. Then,
		\begin{itemize}
			\item{$\rho$ is $(G,X,U)$-cohomologous to a minimal cocycle.}
			\item{$X\times_\rho U$ is ergodic if and only if $X$ is ergodic and $\rho$ is minimal with image $U_\rho = U$.}
		\end{itemize}
	\end{lem}
	The following proposition is the main step in our argument. We show that for every ergodic $G$-system $X$, where $G$ is a torsion free countable abelian group and any symmetric multilinear map $\lambda:G^m\rightarrow S^1$ there exists an extension $Y$ such that $\lambda\in \Spec_m(Y)$.
	\begin{prop}\label{Integration V1}
	    Let $G$ be a torsion free countable abelian group and let $X$ be an ergodic $G$-system. Let $m\in\mathbb{N}$ and suppose that $(\lambda_n)_{n\in\mathbb{N}}\in \text{SML}_m(G,S^1)$ are countably many symmetric multilinear maps. Then, there exists an extension $\pi:(Y,G)\rightarrow (X,G)$ and phase polynomials $Q_n:Y\rightarrow S^1$ of degree $<m+1$ such that $\lambda_n(g_1,...,g_m) = \Delta_{g_1}...\Delta_{g_m} Q_n$. In other words, $\lambda_n\in \Spec_{m}(Y)$ for every $n\in\mathbb{N}$.
	\end{prop}
	\begin{proof}
	    Let $\lambda:G^m\rightarrow (S^1)^\mathbb{N}$ be the multilinear map whose $n$-th coordinate is $\lambda_n$. We prove the claim by induction on $m$. If $m=1$, then $\lambda:G\rightarrow (S^1)^\mathbb{N}$ is a homomorphism. Let $\tau : G\times X\rightarrow (S^1)^\mathbb{N}$ be a minimal cocycle which is cohomologous to $\lambda$ and let $F:X\rightarrow (S^1)^\mathbb{N}$ be such that $\lambda = \tau\cdot \Delta F$. Let $V\leq (S^1)^\mathbb{N}$ denote the image of $\tau$ and consider the extension $Y=X\times_\tau V$. By Lemma \ref{minimal} this extension is ergodic. Let $\imath:V\rightarrow (S^1)^\mathbb{N}$ be the embedding of $V$ in $(S^1)^\mathbb{N}$ and let $Q(x,v)=\imath(v)\cdot F(x)$. Then $\Delta_g Q(x,v) = \tau\cdot \Delta F = \lambda(g)$, which clearly implies that $\Delta_g Q_n = \lambda_n(g)$ where $Q_n$ is the $n$-th coordinate of $Q$, as required. Let $m\geq 2$ and assume inductively that the claim has already been proven for smaller values of $m$. For every $g_m\in G$, the map $(g_1,...,g_{m-1})\mapsto \lambda(g_1,...,g_{m-1},g_m)$ is an element in $\SML_{m-1}(G,S^1)$. By the induction hypothesis, there exists an extension $X_1$ of $X$ and phase polynomials $Q_{g_m}$ of degree $<m$ on $X_1$ such that 
	    \begin{equation}\label{integration}\lambda(g_1,...,g_m) = \Delta_{g_1}...\Delta_{g_{m-1}} Q_{g_m}.
	    \end{equation} In particular, for every $g,g'\in G$ we have \begin{equation}\label{quasicocycle} \frac{Q_{g+g'}}{Q_{g} T_{g} Q_{g'}}\in P_{<m-1}(X_1,S^1).\end{equation}
	    In this case we say that $g\mapsto Q_g$ is quasi-cocycle of order $<m-1$. We claim by induction on $1\leq j \leq m$, that there exist an extension $X_j$ and phase polynomials $Q_g^{(j)}:X_j\rightarrow S^1$ of degree $<m$ such that \begin{equation}\label{integrationj}\lambda(g_1,...,g_m)=\Delta_{g_1}...\Delta_{g_{m-1}} Q_{g_m}^{(j)}
	    \end{equation} and $g\mapsto Q_g^{(j)}:X_j\rightarrow S^1$ is a quasi-cocycle of order $<m-j$. Set $Q_g^{(1)}=Q_g$, the case $j=1$ follows immediately by (\ref{quasicocycle}). Fix $j\geq 2$ and assume inductively that there exist an extension $X_{j-1}$ and phase polynomials $Q_g^{(j-1)}:X_{j-1}\rightarrow S^1$ such that $\lambda(g_1,...,g_m)=\Delta_{g_1}...\Delta_{g_{m-1}} Q_{g_m}^{(j-1)}$ and $g\mapsto Q_g^{(j-1)}$ is a quasi-cocycle of degree $<m-j+1$. For every $g_1,...,g_{m-j+1}\in G$ and every $g,g'\in G$ we have $$\Delta_{g_1}...\Delta_{g_{m-j+1}} \frac{Q^{(j-1)}_{g+g'}}{Q^{(j-1)}_g T_g Q^{(j-1)}_{g'}}=1.$$
	    Therefore, by ergodicity
	    $$k^{(j-1)}_{g,g'}(g_1,...,g_{m-j}):=\Delta_{g_1}...\Delta_{g_{m-j}} \frac{Q^{(j-1)}_{g+g'}}{Q^{(j-1)}_g T_g Q^{(j-1)}_{g'}}$$ is a constant. The map $k:G\times G\rightarrow \SML_{m-j}(G,S^1)$ which sends $(g,g')$ to the symmetric multilinear map $k_{g,g'}$ is a symmetric cocycle (as in Definition \ref{cocycle}). Therefore, it defines an abelian multiplication on the set $B=G\times \SML_{m-j}(G,S^1)$ by $(g,\mu)\cdot (g',\mu') = (g+g', k(g,g')\cdot \mu\cdot \mu')$. We consider the following short exact sequence $$1\rightarrow \SML_{m-j}(G,S^1)\rightarrow B\rightarrow G\rightarrow 1.$$ By Proposition \ref{SML} the group $\SML_{m-j}(G,S^1)$ is divisible and so by Proposition \ref{split} we can find a map $c:G\rightarrow \SML_{m-j}(G,S^1)$ such that $\frac{c(g+g')}{c(g)c(g')} = k(g,g')$. By the induction hypothesis, we can pass to an extension $(X_j,G)$ of $(X_{j-1},G)$ where we can find phase polynomials $Q'_g:X_j\rightarrow S^1$ of degree $<m$ such that $c(g)(g_1,...,g_{m-j}) = \Delta_{g_1}...\Delta_{g_{m-j}} Q'_g$. Now let $Q_g^{(j)}:= Q_g^{(j-1)}\circ \pi_j/Q_g'$ where $\pi_j:X_j\rightarrow X_{j-1}$ is the factor map. Then $g\mapsto Q_g^{(j)}$ is a quasi-cocycle of order $<m-j$. Moreover, since $Q_g'$ are phase polynomials of degree $<m-1$, equation (\ref{integrationj}) holds. This completes the proof by induction. The case $j=m$ implies that we can choose $g\mapsto Q_g$ to be a cocycle, where $Q_g:X_m\rightarrow (S^1)^\mathbb{N}$ are phase polynomial of degree $<m$, $X_m$ is an ergodic extension of $X$ and equation (\ref{integration}) holds. The rest of the proof is the same as in the case where $m=1$. Namely, we can find a minimal cocycle $\tau:G\times X_m\rightarrow V$ which is cohomologous to $(g,x)\mapsto Q_g(x)$. By Lemma \ref{minimal}, the extension $Y=X_m\times_{\tau} V$ is ergodic and the map $Q(x,v) = v\cdot Q(x)$ satisfies that $\Delta_g Q = Q_g$. This implies that that $\lambda(g_1,...,g_m) = \Delta_{g_1}...\Delta_{g_m} Q$ and the proof is complete.
	\end{proof}
	We can finally prove the promised result.
	\begin{thm} \label{divisible_extension}
		Let $G$ be a torsion free countable abelian group and $k\geq 2$. Then for every ergodic system $(X,G)$ there exists an extension $(Y,G)$, such that $X$ is $k$-divisible in $Y$.
	\end{thm}
	\begin{proof}
		Let $X$ be as in the theorem. Fix $k\in\mathbb{N}$, and let $\Spec(X) = \bigcup_{i=1}^{k-1} \Spec_i(X)$. For every $1\leq i \leq k-1$, every $\lambda\in \Spec_i(X_l)$, and every $n\in\mathbb{N}$ choose an $n$-th root $\lambda_n\in SML_i(G,S^1)$ for $\lambda$ (which exists, by Proposition \ref{SML}). Then, by Proposition \ref{Integration V1}, we can find an extension $Y$ such that $\{\lambda_n:\lambda\in \Spec(X),n\in\mathbb{N}\}$ belongs to $\Spec(Y)$. This completes the proof.
	\end{proof}
	As a corollary we conclude the following stronger result for $k=2$.
	\begin{thm} \label{divisible_extension1}
		Let $G$ be a torsion free countable abelian group. Then every ergodic $G$-system $X$ is a factor of a $2$-divisible system.
	\end{thm}
	\begin{proof}
		Let $X$ be as in the theorem. Applying theorem \ref{divisible_extension} iteratively we obtain an increasing sequence of extensions $(X_n,G)$ with the property that $\Spec_1(X_n)$ is divisible in $\Spec_1(X_{n+1})$. Let $Y$ be the inverse limit of $X_n$ and recall that the factor map $\pi:Y\rightarrow X_n$ induces factors $\pi_n:Z_{<2}(Y)\rightarrow Z_{<2}(X_n)$ for every $n\in\mathbb{N}$. It is classical (see \cite[Lemma 8.1]{F&W}) that $Z_{<2}(Y)$ is an inverse limit of the sequence
		$$...\rightarrow Z_{<2}(X_n)\rightarrow Z_{<2}(X_{n-1})\rightarrow...\rightarrow Z_{<2}(X_1)\rightarrow Z_{<2}(X).$$ Let $f$ be an eigenfunction of $Y$, then for every $n\in\mathbb{N}$ and $g\in G$ we have
		$$T_gE(f|Z_{<2}(X_n))=E(T_gf|Z_{<2}(X_n)) =\lambda_g E(f|Z_{<2}(X_n)).$$
		In particular, if $E(f|Z_{<2}(X_n))\not=0$, then $f$ is measurable with respect to $Z_{<2}(X_n)$. Therefore, for sufficiently large $n$, $\Delta_g f\in \Spec_1(X_n)$. Since $\Spec_1(X_n)$ is divisible in $\Spec_1(Y)$ this completes the proof.
	\end{proof}
	We give an example of the theorem above in a simple case. For the sake of simplicity, we will not construct a divisible extension of our initial system $X$, but instead we will define an extension $Y$ where $P_{<2}(Y,S^1)$ is divisible by $2$ (i.e. it contains all of its square roots.).
	\begin{example} \label{rootexample}
	Let $X=(\mathbb{R}/\mathbb{Z},\alpha)$ be an irrational rotation $Tx = x+\alpha$, $\alpha\in\mathbb{R}\backslash\mathbb{Q}$. The maps $\{x\mapsto e(nx):n\in\mathbb{N}\}$ form an orthonormal basis of eigenfunctions for $T:L^2(X)\rightarrow L^2(X)$ (recall that $e(y):=e^{2\pi i y}$). It follows that $P_{<2}(X,S^1) = \{ x\mapsto c\cdot e(nx) : c\in S^1, n\in\mathbb{Z}\}\cong S^1\times\mathbb{Z}$ is not a divisible group. Let $\alpha_1 = \frac{\alpha}{2}$ and consider the new irrational rotation $X_1=(\mathbb{R}/\mathbb{Z},\alpha_1)$. We note that $X_1$ is isomorphic to a group extension of $X$ by $C_2$ and the cocycle $\tau(x) = \alpha_1\cdot F(x+\alpha)\cdot F(x)^{-1}$ where $F$ is any measurable map with $F^2(x)=x$ and the isomorphism $X\times_\tau U\rightarrow X_1$ is given by $(x,u)\mapsto u\cdot \overline{F}(x)$. We follow this procedure and construct a system of extensions $X_n$. Namely, for every $n\geq 2$, let $\alpha_n=\alpha/2^n$ and $X_n=(\mathbb{R}/\mathbb{Z},\alpha_n)$ be the irrational rotation by $\alpha_n$. The map $\pi_n:X_n\rightarrow X_{n-1}$, $\pi_n(x)=x^2$ is a factor map and the sequence $(X_n,\pi_n)$ has an inverse limit which we denote by $Y$. As topological groups, the inverse limit of $X_n$ is isomorphic to the solenoid $(\mathbb{R}\times\mathcal{Z}_2)/\mathbb{Z}$ where $\mathcal{Z}_2 = \{(z_1,z_2,...)\in (\mathbb{R}/\mathbb{Z})^{\mathbb{N}} : 2\cdot z_i=z_{i-1} \forall_{i\geq 2} \}$ are the $2$-adic integers and the group $\mathbb{Z}$ is embedded in $\mathbb{R}\times \mathcal{Z}_2$ by sending $1$ to $(1,(\omega_n))$ where $\omega_n=\frac{1}{2^n}$. Under this identification, the action on $Y$ is given by the rotation $R_{(\alpha,\vec {0})}$ where $\vec{0}$ is the zero element in $(\mathcal{Z}_2,+)$. The factor map $\pi'_n:Y\rightarrow X_n$ is given by $(x,\vec z)\mapsto \frac{x}{2^n}-z_n$. The Pontryagin dual of the solenoid $Y$ is isomorphic to the group $\mathbb{Z}[\frac{1}{2}] = \{\frac{a}{2^n} : a\in \mathbb{Z},n\in\mathbb{N}\}$ and the group $P_{<2}(Y,S^1)\cong S^1\oplus \mathbb{Z}[\frac{1}{2}]$ is divisible by $2$. In other words, every element in $P_{<2}(Y,S^1)$ has a square root in that group.
	\end{example}
	
	\section{Divisible C.L.\ systems are homogeneous} \label{proof}
	We prove Theorem \ref{mainresult}. By Theorem \ref{divisible_extension1} it is enough to show the following result.
	\begin{thm} [Divisible C.L.\ systems are homogeneous] \label{main:thm} Let $G$ be a countable group and let $X$ be an ergodic $2$-divisible $G$-system. Then the action of $\mG(Z_{<3}(X))$ on $Z_{<3}(X)$ is transitive.
	\end{thm}
	We prove Theorem \ref{main:thm} and the properties mentioned in Remark \ref{properties}.
	\begin{proof}
		Let $X$ be as in the Theorem. By Proposition \ref{everything}, we can write $Z_{<3}(X)=Z_{<2}(X)\times_\rho U$ for some compact abelian group $U$ and a cocycle $\rho:G\times Z_{<2}(X)\rightarrow U$. As usual we identify $Z_{<2}(X)$ with a compact abelian group $Z$. Let $\chi\in\hat U$ be a character and $s\in Z$, then by Proposition \ref{everything} again, we can find a character $c_s(\chi):G\rightarrow S^1$ and a measurable map $F_s(\chi):Z\rightarrow S^1$ such that $\Delta_s \chi\circ \rho = c_s(\chi)\cdot \Delta F_s(\chi)$. Since $X$ is $2$-divisible, Proposition \ref{typecircle:prop} implies that for every $s\in Z$ there exists a measurable map $F_s:Z\rightarrow U$ such that $S_{s,F_s}\in \mG(Z_{<3}(X))$ (see Definition \ref{CLgroup}). Since the transformations $S_{1,u}$ for $u\in S^1$ are also in $\mG(Z_{<3}(X))$ the action of this group on $X$ is transitive. This completes the proof of Theorem \ref{mainresult}. Now, let $x_0=(1,1)\in Z\times U$ and $\Gamma$ be the stabilizer of $x_0$ under the action of $\mG(X)$. Then, $$\Gamma = \{S_{1,F} : F\in \text{Hom}(Z,U)\}$$ is a totally disconnected closed subgroup of $\mG(Z_{<3}(X))$. By Theorem \ref{openmap}, the projection map $p:\mG(Z_{<3}(X))\rightarrow \mG(Z_{<3}(X))/\Gamma$ is open and by Theorem \ref{quotient}, $Z_{<3}(X)$ is homeomorphic to $\mG(Z_{<3}(X))/\Gamma$. It follows that $Z_{<3}(X)$ is  isomorphic to $\mG(Z_{<3}(X))/\Gamma$ as $G$-systems, where the action of $g\in G$ on $\mG(Z_{<3}(X))/\Gamma$ is given by left multiplication by $S_{g,\rho(g,\cdot)}$.
	\end{proof}
	We now prove Theorem \ref{double coset}.
	\begin{proof}
	    Let $X$ be as in the theorem and write $Z_{<3}(X) = Z\times_\rho U$ where $Z$ is the Kronecker factor and $\rho:G\times Z\rightarrow U$ is a cocycle into a compact abelian group $U$. Let $(\tilde{Z},H)$ be an extension of $(Z,G)$ with divisible dual (as in Theorem \ref{divisible_extension1}). Let $\pi:\tilde{Z}\rightarrow Z$ be the quotient map and $K:=\ker \pi$. Let $q:Z\rightarrow \tilde{Z}$ be a Borel cross section. The map $\varphi:\tilde{Z}\rightarrow Z\times K$, $z\mapsto (\pi(z), z\cdot q\circ\pi(z)^{-1})$ is a measure-theoretical bijection. Let $\tau:H\times \tilde{Z}\rightarrow K$ be the cocycle $$\tau(h,z) = \frac{T_h z\cdot q\circ \pi(T_hz) }{z\cdot q\circ\pi(z)^{-1}} $$ where $T_h$ denotes the action of $h\in H$ on $\tilde{Z}$. Observe that $\tau$ is invariant to translations by $K$ and so it induces a cocycle $\tau':H\times Z\rightarrow K$ and $\tilde{Z}\cong Z\times_{\tau'} K$. It will be convenient to modify the group $\mathcal{G}$ from Theorem \ref{mainresult}. Let $$\mathcal{G} = \{(s,k,F)\in Z\times K\times \mathcal{M}(\tilde{Z},U) :  \exists c_s:H\rightarrow U \text{ such that } \Delta_s \tau'(h,\pi(z)) = \Delta_h F(z) \}$$ and equip $\mathcal{G}$ with the multiplication $(s,k,F)\cdot (s',k',F') = (ss',kk',FV_{q(s)\cdot k} F')$. We define a topology on $\mathcal{G}$ by letting a sequence $(s_n,k_n,F_n)$ converge to $(s,k,F)$ if $s_n\rightarrow s$ in $Z$, $k_n\rightarrow k$ in $K$ and $F_n\rightarrow F$ in measure. With this topology and multiplication $\mathcal{G}$ is a $2$-step nilpotent polish group. Recall that in the proof of Theorem \ref{mainresult} we show that the projection $\mathcal{G}\rightarrow \tilde{Z}$ is onto. In particular, it follows that $\mathcal{G}$ acts transitively on $X\times K$. Now, let $\Gamma = \{1\}\times \{1\}\times \text{Hom}(\tilde{Z}\rightarrow U)$ and let $\varphi:G\rightarrow \mathcal{G}$ be the homomorphism $\varphi(g) = (T_g1,1,\rho(g,\pi(\cdot))$ where $T_g:Z\rightarrow Z$ denote the action of $G$ on $Z$. Since $\rho\circ\pi$ is invariant to translations by $K$, we have that $\varphi(G)$ commutes with $K$. Moreover, the action of $g\in G$ on $Z_{<3}(X)$ corresponds to multiplication by $\varphi(g)$ under the homeomorphism $Z_{<3}(X)\cong K\backslash \mathcal{G}/\Gamma$.\\
	    
	    It is left to show that $K$ is totally disconnected. We recall the relevant part in the proof of Theorem \ref{divisible_extension1}. For every $\chi\in\hat Z$ and $n\in\mathbb{N}$ we find a homomorphism $\lambda_n\in\hat G$ such that $\lambda_n^n(g)=\Delta_g \chi$. Then, we let $\lambda:G\rightarrow(S^1)^\mathbb{N}$ be the homomorphism whose $n$-th coordinate is $\lambda_n$. Let $\tau$ be a minimal cocycle which is cohomologous to $\lambda$ and $V_1$ be its image. Then, as in the proof of Proposition \ref{Integration V1}, we let $Z_1 = Z\times_\tau V_1$. We first prove that $V_1$ is totally disconnected. Let $F_n:Z\rightarrow S^1$ be any measurable map with $F_n^n = \chi$, then $\lambda_n\cdot \Delta \overline{F_n}$ takes values in $C_n$. Let $F:Z\rightarrow (S^1)^\mathbb{N}$ be the map whose $n$-th coordinate is $F_n$ then $\lambda\cdot\Delta F$ takes values in $\prod_{n} C_n$, which is totally disconnected. Since $\tau$ is minimal, $V_1$ is a closed subgroup of $\prod_n C_n$ and therefore totally disconnected. Now, we continue this process. In each step we construct a Kronecker system $Z_m$ as an extension of $Z_{m-1}$ by a totally disconnected group $V_{m-1}$. The group $\tilde{Z}$ is the inverse limit of the sequence $Z_m$. It follows that  $K$ is the inverse limit of $V_m$. Since $V_m$ is totally disconnected for every $m\in\mathbb{N}$, we conclude that so is $K$.
	\end{proof}
	\subsection{Simple homogeneous spaces}
	For completeness we show that any system with a nilpotent homogeneous structure as in Theorem \ref{mainresult} is an inverse limit of simpler homogeneous spaces in which the stabilizer $\Gamma$ is a discrete subgroup. We will not use this result.
	\begin{defn}
		Let $G$ be a countable abelian group and let $(X,G)$ be a C.L.\ system. We say that $X$ is a simple homogeneous space if the C.L.\ group (Definition \ref{CLgroup}) acts transitively on $X$ and the stabilizer of any $x_0\in X$ is a discrete subgroup.
	\end{defn}
	
	\begin{prop}
		Let $G$ be a countable abelian group and let $(X,G)$ be a C.L.\ system. If $\mG(X)$ acts transitively on $X$ then $X$ is an inverse limit of simple homogeneous spaces.
	\end{prop}
	\begin{proof}
		Let $X$ as in the proposition and write $X=Z\times_\rho U$ where $Z=Z_{<2}(X)$ is the Kronecker factor. By Gleason-Yamabe theorem we can find a decreasing sequence of closed subgroups $K_n\leq U$ such that $\bigcap_{n\in\mathbb{N}}K_n = \{1\}$ and the quotients $L_n=U/K_n$ are Lie groups. Let $\pi_n:U\rightarrow L_n$ be the projection map and let $X_n = Z\times_{\pi_n\circ\rho} L_n$. Since $\mG(X)$ acts transitively on $X$, we have that for every $s\in Z$, there exists a measurable map $F:Z\rightarrow U$ such that $\Delta_s \rho = c\cdot \Delta F$. Observe that if $S_{s,F}\in \mG(X)$ and $\Delta_s \rho = c \cdot \Delta F$ for some $c:G\rightarrow U$, then $\Delta_s \pi_n\circ\rho = \pi_n\circ c \cdot \Delta \pi_n\circ F$ and $S_{s,\pi_n\circ F}\in \mG(X_n)$. As $S_{1,uK_n}$ belongs to $\mG(X_n)$, we conclude that the action of $\mG(X_n)$ on $X_n$ is transitive. Fix any $x_0\in X_n$, then the stabilizer $\Gamma_n$ of $x_0$ is homeomorphic to the discrete group $\hom (Z,U_n)$. Since $X$ is an inverse limit of $X_n$, the claim follows.
	\end{proof}
	
	\section{The structure of a nilpotent system} \label{structure:sec}
	Let $G$ be a countable abelian group and $X$ be a C.L.\ ergodic $G$-system such that the action of $\mG(X)$ on $X$ is transitive. Write $X=\mG(X)/\Gamma(X)$ where $\Gamma(X)$ the stabilizer of some $x_0\in X$. We recall the definition of a group rotation.
	\begin{defn} \label{grouprotation:def}
		Let $G$ be a countable abelian group. We say that a $G$-system $X$ is a group rotation if it is isomorphic to a compact abelian group $K$ and there exists a homomorphism $\varphi:G\rightarrow K$ such $T_g k = \varphi(g)k$ for every $g\in G$ and $k\in K$. 
	\end{defn}
	It is well known (see \cite[Theorem 6.1]{EW}) that the Kronecker factor is the maximal group rotation. 
	\begin{thm} [Maximal property of the Kronecker factor] \label{maximal}
		Let $G$ be a countable abelian group and $X$ be a $G$-system. Then any group rotation factor $Y$ of $X$ is a factor of $Z_{<2}(X)$.
	\end{thm}
	Recall that any C.L.\ system can be written as $X=Z\times_\rho U$ where $Z$ is the Kronecker factor, $U$ is a compact abelian group and $\rho:G\times Z\rightarrow U$ is a cocycle (Proposition \ref{everything}). The following lemma plays an important role in the proof of the limit formula (Theorem \ref{formula} below). We show that if $\mG(X)$ acts transitively on $X$ is transitive, then it is possible to express the groups $Z$ and $U$ in terms of the homogeneous group $\mG(X)$, its commutator $\mG(X)_2$ and the stabilizer $\Gamma(X)$.
	\begin{lem} \label{structure}
		Let $G$ be a countable abelian group and $X=Z\times_\rho U$ be an ergodic C.L.\ $G$-system where $Z$ is the Kronecker factor and suppose that the action of $\mG(X)$ on $X$ is transitive. If $\mG$ is an open subgroup of $\mG(X)$ which contains the embedding of $G$ in $\mG(X)$, then $Z\cong \mG/\mG_2\Gamma$ and $U\cong \mG_2$ where $\Gamma:= \Gamma(X)\cap \mG$ and $\mG_2$ is the closed subgroup generated by the commutators $\{[a,b]:a,b\in \mG\}$\footnote{where $[a,b]=a^{-1}b^{-1}ab$ as usual.}.
	\end{lem}
	\begin{proof}
		First we prove that $\mG/\Gamma\cong \mG(X)/\Gamma(X)$ as measure spaces. To see this observe that the projection $p:\mG(X)\rightarrow \mG(X)/\Gamma(X)$ is an open map (Theorem \ref{openmap}). Therefore, $p(\mG)$ is a $G$ invariant open (and closed) subset of $\mG(X)/\Gamma(X)$, hence by ergodicity $p(\mG)=X$. We conclude that the map $g\Gamma \mapsto g\Gamma(X)$ from $\mG/\Gamma$ to $\mG(X)/\Gamma(X)$ is an isomorphism. In particular, there exists a factor map $\pi:\mG/\Gamma \rightarrow Z$. Direct computation shows that $\mG_2$ acts trivially on $Z$ and $\pi$ factors through $\mG_2$. By Lemma \ref{maximal}, $\pi:\mG/\mG_2\Gamma\rightarrow Z$ is an isomorphism, hence $Z\cong \mG/\mG_2\Gamma$.\\ Let $p:\mG(X)\rightarrow Z$ be the projection map $S_{s,F}\mapsto s$. The group $p(\mG)$ is an open and closed $G$-invariant subgroup of $Z$ and so by ergodicity $p(\mG)=Z$. Choose a Borel cross section $s\mapsto S_{s,F_s}$ as in Theorem \ref{openmap}. We have, 
		$$[S_{g,\sigma(g)},S_{s,F_s}] = S_{1,\frac{\Delta_s \sigma}{\Delta F_s}}$$ and $\frac{\Delta_s \sigma}{\Delta_g F_s}$ is a constant in $U$. We identify $\mG_2$ with the closed subgroup generated by these constants. Suppose by contradiction that $\mG_2\lneqq U$, then there exists a non-trivial character $\chi:U\rightarrow S^1$ such that $\Delta_s \chi\circ \sigma = \Delta \chi\circ F_s$. Theorem \ref{HKtype} implies that factor $Y=Z\times_{\chi\circ\sigma} \chi(U)$ is isomorphic to a group rotation and Theorem \ref{maximal} provides a contradiction.
	\end{proof}
	We need the following weaker notion of divisibility.
	\begin{defn}
		Let $U$ be an abelian group and $n\in\mathbb{N}$. We denote by $U^n:=\{u^n:u\in U\}$ and say that $U$ is $n$-divisible if $U^n=U$. 
	\end{defn}
	As a corollary of the previous lemma we conclude:
	\begin{cor} \label{div}
		Let $G$ be a countable abelian group and $a,b\in\mathbb{Z}$ as in Theorem \ref{Khintchine}. Let $(X,G)$ be an ergodic C.L.\ system and suppose that $\mG(X)$ acts transitively on $X$. Then the commutator subgroup $\mG(X)_2$ is $a$, $b$ and $(b\pm a)$-divisible.
	\end{cor}
	\begin{proof}
		By the previous lemma, we can write $X=Z\times_\sigma U$ where $U=\mG(X)_2$. Fix a number $m\in\{a,b,b\pm a\}$ and suppose by contradiction that $U$ is not $m$-divisible. Then, $U/U^m$ is non-trivial and so, since the characters separate points, there exists a non-trivial character $\chi:U\rightarrow C_m$. Let $s\mapsto S_{s,F_s}$ be a Borel cross section from $Z$ to $\mG(X)$ and let $c_s:G\rightarrow S^1$ such that $\Delta_s \chi\circ \rho = c_s\cdot \Delta F_s$.\\
		Observe that $c_s^m\in B^1(G,X,S^1)$ is an eigenvalue. Since $mG$ is of finite index in $G$, the set $\{c_s:s\in Z\}$ is at most countable. Thus, the group  $Z'=\{s\in Z : \Delta_s \chi\circ\rho\in B^1(G,X,S^1)\}$ is a $G$-invariant open subgroup of $Z$ and by ergodicity, $Z'=Z$. As before, Theorem \ref{HKtype} implies that the extension by $\chi\circ\rho$ is a group rotation and Theorem \ref{maximal} provides a contradiction.
	\end{proof}
	\begin{rem}
		In the previous corollary, since at least one of $a,b,a+b$ is even, the group $\mG(X)_2$ is automatically $2$-divisible.
	\end{rem}
	\section{Limit formula and pointwise convergence} \label{formula:sec}
	We prove the following pointwise convergence for some multiple ergodic averages on a $2$-step homogeneous space where the homogeneous group is the C.L.\ group (Definition \ref{CLgroup}).
	\begin{thm} [Limit formula] \label{formula} 
		Let $X=\mG/\Gamma$ be an ergodic C.L.\ $G$-system, where $\mG$ is the C.L.\ group and suppose that $\mG_2$ is $2$-divisible. Let $\mu_{\mG_2}$ denote the Haar measure on $\mG_2$. Then for every $k\in\mathbb{N}$, $f_1,f_2,...,f_k\in L^\infty(X)$ and $\mu$-almost every $x\in X$ we have,
		\begin{equation} \label{average}\begin{split}\lim_{N\rightarrow\infty} \E_{g\in\Phi_N} \prod_{i=1}^k T_{ig}f_i(x) & =\\ \int_{\mG/\Gamma} \int_{\mG_2} \prod_{i=1}^k f_i(xy_1^i y_2^{\binom{i}{2}}) &d\mu_{\mG_2}(y_2)d\mu(y_1)
		\end{split}
		\end{equation}
		with the abuse of notation that $f(x)=f(x\Gamma)$.
	\end{thm}
	As a corollary we conclude the following result.
	\begin{cor} \label{formulaab}
		Let $a,b\in\mathbb{Z}$. In the settings of Theorem \ref{formula}, choose $k=a+b$, let $h_1,h_2,h_3\in L^\infty(X)$ be any bounded functions and set $f_a=h_1$, $f_b=h_2$,  $f_{a+b}=h_3$ and $f_i=1$ for all $i\not= a,b,a+b$. Then, for $\mu$-almost every $x\in X$ we have,
		\begin{equation} \label{averageab}\begin{split}\lim_{N\rightarrow\infty} \E_{g\in\Phi_N}  T_{ag}h_1(x) T_{bg}h_2(x) T_{(a+b)g}h_3(x)& =\\ \int_{\mG/\Gamma} \int_{\mG_2} h_1(xy_1^ay_2^{\binom{a}{2}})h_2(xy_1^by_2^{\binom{b}{2}}) h_3(xy_1^{a+b}y_2^{\binom{a+b}{2}}) &d\mu_{\mG_2}(y_2)d\mu(y_1).
		\end{split}
		\end{equation}
	\end{cor}
	Note that the assumption that $\mG_2$ is $2$-divisible in Theorem \ref{formula} is necessary. We give a counterexample in the case where this assumption is removed.
	\begin{example}
	Let $G=\mathbb{F}_2^\omega$ be the countable direct sum of the field $\mathbb{F}_2=\{0,1\}$. The map $g\mapsto e^{\pi i g}$ defines an embedding of $G$ in the infinite direct product $Z=\prod_{n=1}^\omega C_2$, where $C_2=\{-1,1\}$ is a discrete group under multiplication. This embedding gives rise to an action of $G$ on $Z$ by $T_g z = e^{\pi i g}\cdot z$ and the system $(Z,G)$ is an ergodic Kronecker system. Let $\sigma:\mathbb{F}_p^\omega \times \prod_{n=1}^\omega C_2\rightarrow C_2$ be the cocycle $\sigma(g,x)=\prod_{i=1}^\infty x_i^{g_i}\cdot (-1)^{\binom{g_i}{2}}$. The system $X=Z\times_{\sigma} C_2$ is an ergodic $G$-system.\footnote{To prove ergodicity one can express any measurable bounded function as $f(x,y)=\sum_{\chi\in \hat Z,\tau\in \widehat{C_2}} a_{\chi,\tau} \chi(x)\tau(y)$. If $f$ is invariant one can use the uniqueness of the Fourier series to deduce that $a_{\chi,\tau}=0$ unless $\chi$ and $\tau$ are the trivial characters.} Since $\sigma$ is a phase polynomial of degree $<2$, it is not hard to show that $X=\mG(X)/\Gamma$ where $\mG(X)$ is the Host-Kra group of $X$ and $\Gamma=\{S_{1,p} : p\in\hat Z\}$. Moreover $\mG(X)_2 = C_2$ is not a $2$-divisible group. Now let $k=2$, $f_1=1$ and $f_2(x,y)=y$. Since, every element of $\mG$ is of order $2$, the integral in equation (\ref{average}) equals to  $\int_{C_2} u^{\binom{2}{2}}y d\mu_{C_2}(u) = 0$. On the other hand $\lim_{N\rightarrow\infty} E_{g\in \Phi_N} T_gf_1(x,y)\cdot T_{2g} f_2(x,y) = y$. Thus, equation (\ref{average}) fails. Note that a similar example where $2$ is replaced by any odd integer $n$, will not give a counterexample for (\ref{average}), because odd $n$'s divide $\binom{n}{2}$.
	\end{example}
	
	In order to prove Theorem \ref{formula}, we follow an argument by Bergelson Host and Kra \cite{BHK}. Let,
	\begin{align*}\imath:\mG\times \mG \times \mG_2 \rightarrow \mG^{k+1},&\\\imath(g,g_1,g_2) = (g,gg_1,gg_1^2g_2,...&,gg_1^kg_2^{\binom{k}{2}}).
	\end{align*}
	We denote by $\tilde{\mG}$ the image of $\imath$. In \cite{Leib} Leibman proved that $\tilde{\mG}$ is a $2$-step nilpotent group. The subgroup $\tilde{\Gamma} = \imath(\Gamma\times \Gamma\times \{e\})$ is a closed subgroup of $\tilde{\mG}$ and the quotient space $\tilde{X}= \tilde{\mG}/\tilde{\Gamma}$ is compact. Let $\tilde{\mu}$ be the Haar measure on this space. We define an action of $G\times G$ on $(\tilde{X},\tilde{\mu})$ by left multiplication with $g^\triangle:=(g,g,...,g)$ and $g^\star = (1,g,g^2,...,g^k)$, where $g$ is identified with the measure-preserving transformation $T_g:X\rightarrow X$ in $\mG(X)$. In Lemma \ref{ergodic} below we prove that this action is uniquely ergodic. Assuming this for now, we fix $x\in X$ and consider the compact polish space $$\tilde{X}_x:=\{(x_1,x_2,..,x_k)\in X^k : (x,x_1,x_2,...,x_k)\in \tilde{X}\}.$$
	Bergelson Host and Kra showed that the group $\tilde{\mG}^\star =\{(g_1,g_1^2g_2,...,g_1^kg_2^{\binom{k}{2}}) : g_1\in \mG, g_2\in \mG_2\}$ acts transitively on this space and $\tilde{X}_x\cong \tilde{\mG}^\star/\tilde{\Gamma}^\star$ where $\tilde{\Gamma}^\star = \{(\gamma,\gamma^2 ,...,\gamma^k): \gamma\in \Gamma\}$. Observe that since $\imath$ is injective, it induces an isomorphism of $G$-systems, $\tilde{\imath} :\mG/\Gamma\times \mG_2\rightarrow \tilde{X}_x$ where the action of $G$ on $\mG\times \mG_2$ is given by $T_g(y_1,y_2) = (g[g,x]y_1,[g,y_1]y_2)$.\\
	We continue assuming that the action of $G\times G$ on $\tilde{X}$ is uniquely ergodic. Let $\tilde{\mu}_x$ be the Haar measure on $\tilde{X}_x$, Bergelson Host and Kra \cite{BHK} proved:
	\begin{lem} \label{measure}
		$$\tilde{\mu} = \int_X \delta_x\otimes \tilde{\mu}_x d\mu(x).$$
	\end{lem}
	We can now prove Theorem \ref{formula}.
	\begin{proof}
		Since continuous functions are dense in $L^\infty(X)$, it is enough to prove the theorem for continuous $f_1,f_2,...,f_k$. Let
		$F:\mG/\Gamma^k\rightarrow \mathbb{C}$, $F(x_1,x_2,...,x_k)=f_1(x_1)\cdot f_2(x_2)\cdot...\cdot f_k(x_k)$, we can write average  (\ref{average}) as
		$$\mathbb{E}_{g\in \Phi_N} (T_{g}\times T_{2g}\times...\times T_{kg})F(x,x,...,x).$$ Recall that every element in the orbit of $(x,x,...,x)$ with respect to the transformation $T_{g}\times T_{2g}\times...\times T_{kg}$ belongs to $\tilde{X}_x$. Thus, by the pointwise ergodic theorem average (\ref{average}) converges pointwise everywhere to a function $\phi(x)$ on $X$. Let $f$ be any continuous function on $X$. Then,	
		$$\int f(x)\phi(x)d\mu(x) =\lim_{N\rightarrow\infty} \int \mathbb{E}_{g\in \Phi_N} f(x)\cdot \prod_{i=1}^k f_i(T_{ig}x)d\mu(x).$$
		Since $\mu$ is $G$-invariant, the above equals to
		\begin{equation}\label{ave}\lim_{N\rightarrow\infty}\int \mathbb{E}_{g,h\in \Phi_N} f(T_hx)\prod_{i=1}^kf_i(T_{ig+h}x)d\mu(x).
		\end{equation}
		Recall that we assume that the action of $G\times G$ by $T_{g^\star}$ and $T_{h^\triangle}$ is uniquely ergodic. Since $(x,x,...,x)$ belongs to $\tilde{\mG}/\tilde{\Gamma}$, we conclude by the pointwise ergodic theorem that (\ref{ave}) converges everywhere to 
		$$\int_{\tilde{X}}f(x_0)\prod_{i=1}^k f(x_i)d\tilde{\mu}(x_0,x_1,...,x_k)$$ which by Lemma \ref{measure} equals to $$\int_X f(x) \left(\int_{\tilde{X}_x} \prod_{i=1}^k f_i(x_i)d\tilde{\mu}_x(x_1,...,x_k)\right)d\mu(x).$$
		As this holds for every continuous function $f$, we conclude that $$\phi(x)=\int_{\tilde{X}_x} \prod_{i=1}^k f_i(x_i)d\tilde{\mu}_x(x_1,...,x_k) = \int_{\mG/\Gamma}\int_{\mG_2} f_1(xy_1)f_2(xy_1^{2}y_2)\cdot...\cdot f_3(xy_1^ky_2^{\binom{k}{2}})d\mu_{\mG_2}(y_2)d\mu(y_1)$$ for $\mu_\mG$-a.e. $x\in \mG$, as required.
	\end{proof}
	By Parry \cite{Parryb} an ergodic action on $\tilde{G}/\tilde{\Gamma}$ is uniquely ergodic (see also a proof by Leibman \cite[Theorem 2.19]{Leib2} that holds in this generality). Therefore, in order to complete the proof of Theorem \ref{formula} it is left to prove that the action of $G\times G$ is ergodic.
	\begin{rem} \label{property H}
	A key component in the argument of Bergelson Host and Kra \cite{BHK} is a result of Green \cite{G} which was generalized by Leibman \cite{Leib2} to nilsystems $(\mG/\Gamma, R_a)$ satisfying the property that $\mG$ is generated by its connected component and $a$. This result asserts that the action of $R_a$ on $\mG/\Gamma$ is ergodic if and only if the induced action of $R_a$ on $\mG/\mG_2\Gamma$ is ergodic. Unfortunately, the connected component of $\tilde{\mG}$ and the transformations $(g^{\triangle})_{g\in G}$ and $(g^\star)_{g\in G}$ may not generate $\tilde{\mG}$. The main observation in our proof is that one can still apply Green's theorem for the nilsystem $\tilde{\mG}/\tilde{\Gamma}$ if the group $\mG$ is the C.L.\ group (Definition \ref{CLgroup}). Lemma \ref{structure} plays an important role in the proof of this observation.
	\end{rem}
	The following lemma is a corollary of Lemma \ref{structure}.
	\begin{lem} \label{open}
		Let $\tilde{\mG}$ as in the proof of Theorem \ref{formula}. If $V\leq \tilde{\mG}$ is an open subgroup which contains $g^\triangle$  and $g^\star$ for all $g\in G$, then
		$$V_2 = \{(g,gg_1,gg_1^2g_2,...,gg_1^kg_2^{\binom{k}{2}}):g,g_1,g_2\in \mG_2\}.$$
	\end{lem}
	\begin{proof}
		Let $\imath:\mG\times \mG\times \mG_2\rightarrow \tilde{\mG}$ be as in the proof of Theorem \ref{formula}. Let $\mL,\mL'\leq \mG$ be open subgroups such that $\mL\times \mL'\times \{e\}\leq i^{-1}(V)$. Since $g^\triangle\in V$ we can assume that $g\in \mL$ and since $g^\star\in V$ that $g\in \mL'$. By shrinking $\mathcal{L},\mathcal{L}'$, we may assume that $g\in \mL=\mL'$. By Lemma \ref{structure} we have that $\mL_2 = \mG_2$. Let $g,g_1,g_2\in \mG_2$:\\
		For every $s_1,s_2\in\mL$, $(s_1,s_1,...,s_1)$ and $(s_2,s_2,...,s_2)$ belong to $V$ and therefore $$(g,g,...,g)\in V_2.$$
		For every $t_1,t_2\in\mL$, we have that $(t_1,t_1,...,t_1)$ and $(e,t_2,t_2^2,...,t_2^k)$ belong to $V$. Since the commutator is a bilinear map, we conclude that $$(e,g_1,g_1^2,...,g_1^k)\in V_2.$$
		Finally, for every $r_1,r_2\in \mL$, $(e,r_1,r_1^2,...,r_1^k)$ and $(e,r_2^1,r_2^2,...,r_2^k)$ belong to $V$ and $$(e,[r_1,r_2]^{1^2},[r_1,r_2]^{2^2},...,[r_1,r_2]^{k^2})\in V_2.$$ Since $(r_2,r_2,....,r_2)$ also belongs to $V$, $(e,[r_1,r_2],[r_1,r_2]^2,...,[r_1,r_2]^k)\in V_2$. We conclude that $$(e,[r_1,r_2]^{1^2-1},[r_1,r_2]^{2^2-2},....,[r_1,r_2]^{k^2-k})\in V_2$$ and since $\mG_2$ is $2$-divisible, $$(e,e,g_2,g_2^{\binom{2}{1}},...,g_2^{\binom{k}{2}})\in V_2.$$
		Combining everything we see that $V_2=\{(g,gg_1,gg_1^2g_2,...,gg_1^kg_2^{\binom{k}{2}}):g,g_1,g_2\in \mG_2\}$ as required.
	\end{proof}
	\begin{cor}\label{kron_ergodic}
		The induced action of $g^\triangle$ and $g^\star$ on $\tilde{\mG}/\tilde{\mG}_2\Gamma$ is ergodic.
	\end{cor}
	\begin{proof}
		The map $\imath$ induces a factor map $\mG/\mG_2\Gamma \times \mG/\mG_2\Gamma\rightarrow \tilde{\mG}/\tilde{\mG}_2\Gamma$. The lift of $g^\triangle$ and $g^\star$ in $\mG/\mG_2\Gamma \times \mG/\mG_2$ corresponds to $T_g\times T_g$ and $Id\times T_g$ respectively. Since $\mG/\Gamma$ is ergodic the claim follows.
	\end{proof}
	We can finally prove the ergodicity of the action.
	\begin{lem} \label{ergodic}
		The action of $G\times G$ on $\tilde{\mG}/\tilde{\Gamma}$ by $g^\triangle$ and $g^\star$ is ergodic.
	\end{lem}
	\begin{proof}
		We follow an argument of Parry \cite{Parryc}. Let $f:\tilde{\mG}/\tilde{\Gamma}\rightarrow S^1$ be an invariant function. The compact abelian group $\tilde{\mG}_2$ acts on $L^2(\tilde{\mG}/\tilde{\Gamma})$. Therefore, we can find eigenfunctions $f_\lambda$, such that $f=\sum_\lambda a_\lambda f_\lambda$ where $a_\lambda\in\mathbb{C}$ and $\lambda$ is a character of $\tilde{\mG}_2$. By the uniqueness of the decomposition, it follows that $f_\lambda$ is also an eigenfunction with respect to the action of $g^\triangle$ and $g^\star$. By Corollary \ref{kron_ergodic} we can assume that $f_\lambda$ takes values in $S^1$. Fix $u\in \tilde{\mG}$, and let $h=g^\triangle$ or $h=g^\star$. Then,
		$$f_\lambda(uhx) = f_\lambda([u^{-1},h^{-1}]hux) = \lambda([u^{-1},h^{-1}])f_\lambda(hux)=\lambda([u^{-1},h^{-1}])c_h f_\lambda(ux)$$
		for some constant $c_h\in S^1$. Therefore, the function $\Delta_u f_\lambda(x)$ is an eigenfunction with respect to the action of $G\times G$ and is invariant under the action of $\tilde{\mG_2}$. By Corollary \ref{kron_ergodic} and Lemma \ref{sep:lem} the set $\{\Delta_u f_\lambda : u\in\tilde{\mG}\}$ is countable modulo constants. It follows that $V_\lambda:=\{u\in \tilde{\mG}: \Delta_u f_\lambda \text{ is a constant}\}$ is an open subgroup. Observe that $u\mapsto \Delta_u f_\lambda$ is a homomorphism from $V_\lambda\rightarrow S^1$ and is therefore trivial on the commutator subgroup $(V_{\lambda})_2$ which by Lemma \ref{open}, equals to $\tilde{\mG}_2$. We conclude that $f$ is invariant under the action of $\tilde{\mG}_2$, and by Corollary \ref{kron_ergodic} is a constant.
	\end{proof}
	\section{Proof of the Khintchine type recurrence} \label{Khintchine:sec}
	In this section we finish the proof of the Khintchine type recurrence (Theorem \ref{Khintchine}). First, we prove a lifting lemma which allows us to replace any system $(X,G)$ with an extension $(Y,H)$.  
	\begin{lem} \label{lift}
		Let $G$ be a countable abelian group and $(X,T_g)$ be a $G$-system. Let $\varphi:H\rightarrow G$ be a surjective homomorphism and $(Y,S_h)$ be an $H$-extension of $X$ with a factor map $\pi:Y\rightarrow X$. Let $\psi:G\times X\rightarrow\mathbb{C}$ be a measurable and suppose that for every F{\o}lner sequence $\Psi_N$ of $H$ we have that $$ \mathbb{E}_{h\in \Psi_N} \psi(\varphi(h),\pi(y))$$ converges in $L^2(Y)$ as $N\rightarrow\infty$. Then the limit equals to $\phi\circ\pi$, where $\phi:X\rightarrow \mathbb{C}$ satisfies
		$$\phi =  \lim_{N\rightarrow\infty}\mathbb{E}_{g\in \Phi_N} \psi(g,x)$$ for every F{\o}lner sequence $\Phi_N$ of $G$.
		In particular, this limit exists in $L^2(X)$.
	\end{lem} 
	Note that we will apply this lemma with $$\psi(g,x) = T_{ag} f_1(x)\cdot T_{bg} f_2(x)\cdot T_{(a+b)g} f_3(x)$$ where $f_1,f_2,f_3\in L^\infty(X)$ in order to deduce the converges of average (\ref{average3}), but it is necessary to prove the result in this generality.\\
	Let $G$ be a countable abelian group. An invariant mean on $G$ is an additive measure $\mu$ on $G$ which is invariant to translations by every $g\in G$.
	\begin{proof}
    Let $\Phi_N$ be a F\o lner sequence for $G$. It is well known (see \cite{Folner}) that there exists an invariant mean $\mu_G$ on $G$ with the property that: For every sequence $\xi:G\rightarrow\mathbb{C}$, if  $\lim_{N\rightarrow\infty} \mathbb{E}_{g\in\Phi_N} \xi(g)$ exists, then it equals to $\int_G \xi(g) d\mu_G(g)$. Since every (discrete) abelian group is amenable, we can find an invariant mean on $H$ with the property that $\mu_H(A) = \mu_G (\varphi(A))$ for every $A\leq H$. Again by a theorem of F{\o}lner, there exists a F{\o}lner sequence $\Psi_N$ on $H$ such that if  $\lim_{N\rightarrow\infty} \mathbb{E}_{h\in\Psi_N} \xi(\varphi(h))$ exists, then it equals to 
    \begin{equation}\label{reduction}\int_H \xi(\varphi(h))d\mu_H(h)=\int_G \xi(g)f\mu_G(g).
    \end{equation}
    We now prove the claim in the lemma: Let $\psi:G\times X\rightarrow\mathbb{C}$ and suppose that $\mathbb{E}_{h\in \Psi_N} \psi(\varphi(h),\pi(y))$ converges in $L^2(Y)$. Since $y\mapsto \psi(\varphi(h),\pi(y))$ are measurable with respect to the factor $X$, we can find $\phi:G\times X\rightarrow\mathbb{C}$ such that  $$\lim_{N\rightarrow\infty}\mathbb{E}_{h\in \Psi_N} \psi(\varphi(h),\pi(y))= \phi\circ \pi.$$ 
    Now let $\xi(g)=\|\psi(g,x)-\phi(x)\|_{L^2(X)}$. By assumption, $\mathbb{E}_{h\in \Psi_N}\xi(\varphi(h))$ converges to zero as $N\rightarrow\infty$. From this and equation (\ref{reduction}) we conclude that for every F{\o}lner sequence $\Phi_N$ of $G$, $\mathbb{E}_{g\in \Phi_N}\xi(g)$ also converges to zero. This completes the proof.
	\end{proof}
	The rest of the proof follows an argument of Frantzikinakis \cite{Fran}.
	\begin{proof}[Proof of Theorem \ref{Khintchine}]
		Let $(X,\mathcal{B},\mu,G)$ be an ergodic $G$-system and let $0\not=a,b\in\mathbb{Z}$ be as in Theorem \ref{Khintchine}. We first prove the theorem in the case where $a$ and $b$ are coprime.\\
		For every $f\in L^\infty(X)$ let $\tilde{f}=E(f|Z_{<3}(X))$. Recall that the Kronecker factor is a group rotation, and denote by $\alpha_g\in Z_{<2}(X)$ the rotation defined by $g\in G$. Then,\\
		\textbf{Claim:} For every continuous function $\eta:X\rightarrow \mathbb{R}^+$ which is measurable with respect to the Kronecker factor (i.e. $\eta=\tilde{\eta}$) and $f_1,f_2,f_3\in L^\infty (X)$ we have
		$$\lim_{N\rightarrow\infty}\mathbb{E}_{g\in \Phi_N} \eta(\alpha_g) T_{ag}f_1 \cdot T_{bg}f_2 \cdot T_{(a+b)g} f_3 = \lim_{N\rightarrow\infty} \mathbb{E}_{g\in \Phi_N} \eta(\alpha_g) T_{ag}\tilde{f}_1 \cdot T_{bg}\tilde{f}_2 \cdot T_{(a+b)g} \tilde{f}_3.$$
		\begin{proof}
			By approximating $\eta$ by linear combinations of eigenfunctions, we see that it is enough to prove the claim in the case where $\eta$ is a character of $\hat Z$. Since $a$ and $b$ are coprime, we can choose $s,t\in\mathbb{Z}$ such that $\eta^{sa}\cdot \eta^{tb}=\eta$. Since $\eta$ is an eigenfunction, it is measurable with respect to $Z_{<3}(X)$ and $E(\eta^s f_1|Z_{<3}(X))=\eta^sE(f_1|Z_{<3}(X))$, $E(\eta^t\cdot f_2|Z_{<3}(X))=\eta^tE(f_2|Z_{<3}(X))$. Thus, by applying Proposition \ref{char} for $\eta^s\cdot f_1, \eta^t\cdot f_2$ and $f_3$ the claim follows.
		\end{proof}
		Assume by contradiction that Theorem \ref{Khintchine} fails. Then one can find $\varepsilon>0$ and a F{\o }lner sequence $\Phi_N$ for $G$ such that \begin{equation} \label{fail}\mu(A\cap T_{ag}A\cap T_{bg}A\cap T_{(a+b)g}A)<\mu(A)^4 -\varepsilon
		\end{equation} for every $g\in\bigcup_N \Phi_N$.\\
		By Theorem \ref{mainresult}, we can find a surjective homomorphism $\varphi:H\rightarrow G$ and an $H$-extension $(\tilde{X},H)$ of $(X,G)$  such that the factor $Y=Z_{<3}(\tilde{X})$ is a C.L.\ system and $Y=\mG(Y)/\Gamma$. Note that since every extension in the proof of Theorem \ref{mainresult} only extends the Kronecker factor of $X$ we have by Lemma \ref{structure} that $\mG(Y)_2 = \mG(Z_{<3}(X))_2$.
		
		Let $f\in L^\infty(X)$, we can push-forward $\tilde{f}$ to a function on $Z_{<3}(X)$ and then let $f^\star$ denote the pullback of this function to $Y$. Let $\Phi^H_N$ be any F{\o}lner sequence for $H$.\\
		\textbf{Claim:} The average
		$$\mathbb{E}_{h\in \Phi^H_N} \eta^\star(\beta_h) S_{ah}f^\star_1(y) S_{bh}f^\star_2(y) T_{(a+b)h}f^\star_3(y)$$ converges to $$\int_Y\int_{G_2(Y)}\eta^\star(y_1)f_1^\star(yy_1^a y_1^{\binom{a}{2}})f_2^\star(yy_1^by_2^{\binom{b}{2}})f_3^\star (yy_1^{a+b}y_2^{\binom{a+b}{2}})d\mu_{G_2(Y)}(y_2)d\mu_{Y}(y_1)$$ 
		where $\beta_h\in Z_{<2}(Y)$ denotes the rotation defined by $h\in H$ on the Kronecker factor of $Y$.
		\begin{proof}
		   Since $\eta$ is measurable with respect to the Kronecker factor, it is enough to prove the claim in the case where $\eta$ is a character of $Z_{<2}(X)$. As in the proof of the previous claim we can find $s$ and $t$ such that $\eta^{sa}\cdot \eta^{tb}=\eta$. Now, we can apply Corollary \ref{formulaab} with $(\eta^\star)^s\cdot f_1^\star , (\eta^\star)^t \cdot f_2^\star$ and $f_3^\star$. This completes the proof.
		\end{proof} Set $f_1=f_2=f_3=1$ in the claim above and apply lemma \ref{lift}. We conclude that
		\begin{equation} \label{avg1}
		\lim_{N\rightarrow\infty}\mathbb{E}_{g\in\Phi_N}\eta(\alpha_g) = \lim_{N\rightarrow\infty} \mathbb{E}_{h\in \Phi_N^H}\eta^\star(\beta_h) = 1.
		\end{equation}
		Now let $\eta$ be arbitrary and set $f=f_0=f_1=f_2=f_3= 1_A$, we conclude that the average
		$$\mathbb{E}_{g\in\Phi^H_N}\eta^\star(\beta_h) \int_Y f^\star(y)\cdot S_{hg} f^\star(y) \cdot S_{bg}f^\star(y) \cdot S_{(a+b)h}f^\star(y) d\mu_Y(y)$$ 
		converges to
		$$\int_Y\int_Y\int_{G(Y)_2}\eta^\star(y_1)f^\star(y)f^\star(yy_1^a y_1^{\binom{a}{2}})f^\star(yy_1^by_2^{\binom{b}{2}})f^\star (yy_1^{a+b}y_2^{\binom{a+b}{2}})d\mu_{G(Y)_2}(y_2)d\mu_Y(y_1)d\mu_Y(y).$$
		This holds for every continuous function $\eta$. Since continuous functions are dense in $L^2$, the above holds for every bounded $Z_{<2}(X)$-measurable $\eta$. Let $\delta>0$ and let $B(\mG(Y)_2,\delta)$ denote the union of all balls of radius $\delta$ with center in $\mG(Y)_2$. We consider the indicator function $\eta=\frac{1}{\mu(B(\mG(Y)_2,\delta)}\cdot 1_B(\mG(Y)_2,\delta)$. Since translations are continuous in $L^2$, taking a limit as $\delta\rightarrow 0$ the above is arbitrarily close to
		$$\int_Y \int_{\mG(Y)_2\times \mG(Y)_2} f^\star(y)f^\star(yy_1^a y_1^{\binom{a}{2}})f^\star(yy_1^by_2^{\binom{b}{2}})f^\star (yy_1^{a+b}y_2^{\binom{a+b}{2}}) d\mu_{\mG(Y)_2\times \mG(Y)_2}(y_1,y_2)d\mu_Y(y).$$
		We integrate everything to get this equals to
		$$\int_Y \int_{\mG(Y)_2^3} f^\star(yy')f^\star(yy'y_1^a y_2^{\binom{a}{2}})f^\star(yy'y_1^b y_2^{\binom{b}{2}})f^\star(yy'y_1^{a+b} y_2^{\binom{a+b}{2}})d\mu_{\mG(Y)_2^3}(y',y_1,y_2) d\mu_Y (y).$$
		By Proposition \ref{computation} we can write the above integral as
		$$\int_Y \int_{\mG(Y)_2^3} f^\star(u^{a+b}y)f^\star(t\cdot u^{b-a}y)f^\star(t\cdot v^{b-a}y)f^\star(v^{a+b}y)d\mu_{\mG(Y)_2^3}(t,u,v) d\mu_Y(y).$$ This clearly equals to
		$$\int_Y \int_{\mG(Y)_2} \left(\int_{\mG(Y)_2} f^\star(u^{a+b}y)f^\star(tu^{b-a}y)d\mu_{\mG(Y)_2}(u)\right)^2 d\mu_{\mG(Y)_2}(t) d\mu_Y(y).$$
		We take the square outside and change variables, the above is greater or equal to $$\int_Y \int_{\mG(Y)_2}\left( f^\star(ty)dm_{\mG(Y)_2}(t)\right)^4d\mu_Y(y)=\left(\int_Y f^\star(x)d\mu_Y(y)\right)^4=\mu(A)^4.$$
		We conclude by Lemma \ref{lift} that for every $\varepsilon>0$, for sufficiently large $N$ and a suitable $\eta$ we have,
		\begin{equation} \label{bigger} \mathbb{E}_{g\in\Phi_N}\eta(a_g)\mu(A\cap T_{ag}A\cap T_{bg}A\cap T_{(a+b)g}A)>\mu(A)^4-\varepsilon/2.
		\end{equation} 
		Therefore if $a$ and $b$ are co-prime, equations (\ref{avg1}) and (\ref{bigger}) contradict equation (\ref{fail}) and the claim follows.\\
		Now let $a$ and $b$ be arbitrary non-zero integers and write $a=a'd, b=b'd$ where $a'$ and $b'$ are coprime. Since $aG$ and $bG$ are of finite index in $G$ so is $dG$ and so $X$ has finitely many ergodic components with respect fo $dG$ with the same Kronecker factor. Choose $\eta$ as before (the same $\eta$ for all ergodic components) and let $\mu = \frac{1}{k}\sum_{i=1}^k \mu_i$. Since $a',b'$ are coprime by equation (\ref{bigger}) we have  $$\mathbb{E}_{g\in \Phi_N} \eta(\alpha_g)\mu_i(A\cap T_{a'g}A\cap T_{b'g}A\cap T_{(a'+b')g}A)>\mu_i(A)^4-\varepsilon/2$$ for all $1\leq i \leq k$. Since $\mathbb{E}_{1\leq i \leq k}(\mu_i (A)^4) \geq \mu(A)^4$, we conclude as before that the set
		$$\{g\in dG : \mu(A\cap T_{a'g}\cap T_{b'g}A\cap T_{(a'+b')g}A)>\mu(A)^4-\varepsilon\}$$ is syndetic. Since $dG$ is of finite index in $G$ this is equivalent to the claim in the theorem. 
	\end{proof}
	\appendix
	\section{Abelian extensions and phase polynomials} \label{poly:appendix}
	In this section we summarize previous results related to abelian extensions and phase polynomials.\\
	
	The following proposition were proved by Host and Kra for $\mathbb{Z}$-actions \cite{HK}. The same argument holds for all countable abelian groups (for details see \cite{ABB}).
	\begin{prop}  \label{HKextension}
		Let $k\geq 1$, let $G$ be a countable abelian group and let $X$ be an ergodic $G$-system. Then $Z_{<k+1}(X)$ is an abelian extension of $Z_{<k}(X)$.
	\end{prop}
	It is natural to ask under which conditions an abelian extension of a system of order $<k$ is of order $<k+1$. To answer this we need the following definitions.
	\begin{defn} [Cubic measure spaces] \cite[Section 3]{HK} Let $G$ be a countable abelian group and $X=(X,\mathcal{B},\mu,G)$ be a $G$-system. For each $k\geq 0$ we define a system $X^{[k]} =(X^{[k]},\mathcal{B}^{[k]},\mu^{[k]},G^{[k]})$ where $X^{[k]}=X^{2^k}$ is the product of $2^k$ copies of $X$,  $\mathcal{B}^{[k]}=\mathcal{B}^{2^k}$ and $G^{[k]}=G^{2^k}$ acting on $X^{[k]}$ in the obvious manner. We define the cubic measures $\mu^{[k]}$ and $\sigma$-algebras $\mathcal{I}_k\subseteq \mathcal{B}^{[k]}$ inductively. $\mathcal{I}_0$ is defined to be the $\sigma$-algebra of invariant sets in $X$, and $\mu^{[0]}:=\mu$. Once $\mu^{[k]}$ and $\mathcal{I}_k$ are defined, we identify $X^{[k+1]}$ with $X^{[k]}\times X^{[k]}$ and define $\mu^{[k+1]}$ by the formula 
		$$\int f_1(x)f_2(y) d\mu^{[k+1]}(x,y) = \int E(f_1|\mathcal{I}_k)(x)E(f_2|\mathcal{I}_k)(x) d\mu^{[k]}(x).$$
		For $f_1,f_2$ functions on $X^{[k]}$ and $E(\cdot|\mathcal{I}_k)$ the conditional expectation, and $\mathcal{I}_{k+1}$ being the $\sigma$-algebra of invariant sets in $X^{[k+1]}$.
	\end{defn}
	This leads to the following generalization of Definition \ref{CLgroup}.
	\begin{defn}[The Host-Kra group for a system of order $<k$.] \label{HKgroup}
		Let $G$ be a countable abelian group and $k\geq 1$. We define $\mG(X)$ to be the group of measure preserving transformations $t:X\rightarrow X$ which satisfies the following property: For every $l>0$, the transformation $t^{[l]}:X^{[l]}\rightarrow X^{[l]}$, $t^{[l]}(x_\omega)_{\omega\in 2^k} = (tx_\omega)_{\omega\in 2^k}$, leaves the measure $\mu^{[l]}$ invariant and acts trivially on the invariant $\sigma$-algebra $I_{l}$. 
	\end{defn}
	Equipped with the topology of convergence in measure $\mG(X)$ is a $(k-1)$-step nilpotent locally compact polish group \cite[Corollary 5.9]{HK}.\\
	
	The cubic measure spaces of Host and Kra also lead to the following definition.
	\begin{defn} [Functions of type $<k$]  \label{type:def} Let $G$ be a countable abelian group, let $X=(X,\mathcal{B},\mu,G)$ be a $G$-system. Let $k\geq 0$ and let $X^{[k]}$ be the cubic system associated with $X$.
		\begin{itemize}
			\item{For each measurable $f:X\rightarrow U$, we define a measurable map $d^{[k]}f:X^{[k]}\rightarrow U$, $$d^{[k]}f((x_w)_{w\in \{-1,1\}^k}):=\prod_{w\in \{-1,1\}^k}f(x_w)^{\text{sgn}(w)}$$ where $\text{sgn}(w)=w_1\cdot w_2\cdot...\cdot w_k$.}
			\item {Similarly, for each measurable $\rho:G\times X\rightarrow U$ we define a measurable map $d^{[k]}\rho:G\times X^{[k]}\rightarrow U$ by
				$$d^{[k]}\rho(g,(x_w)_{w\in \{-1,1\}^k}):=\prod_{w\in \{-1,1\}^k} \rho(g,x_w)^{\text{sgn}(w)}.$$}
			\item {A function $\rho:G\times X\rightarrow U$ is said to be a function of type $<k$ if $d^{[k]}\rho$ is a $(G,X^{[k]},U)$-coboundary.}
		\end{itemize}
	\end{defn}
	We now answer exactly when an abelian extension of a system of order $<k$ is of order $<k+1$.
	\begin{thm} \label{HKtype}
		Let $k,m\geq 1$ and let $G$ be a countable abelian group. Let $(X,G)$ be an ergodic $G$-system of order $<k$ and $\rho:G\times X\rightarrow U$ be a cocycle into some compact abelian group $U$. Then,
		\begin{itemize}
			\item {$X\times_{\rho} U$ is of order $<k+1$ if and only if $\rho$ is of type $<k$.}
			\item{If $\rho$ is of type $<k-1$, then $X\times_{\rho} U$ is of type $<k$.}
		\end{itemize}
	\end{thm}
	\begin{proof}
		The first claim is proved in \cite[Proposition 6.4]{HK} and the second in \cite[Proposition 7.6]{HK} for $\mathbb{Z}$-actions. The general case follows by the same argument.
	\end{proof}
	In particular this implies that the C.L.\ factor of an ergodic $G$-system is an abelian extension of the Kronecker factor by a cocycle of type $<2$. The following definition is closely related to the Conze-Lesigne equations in Definition \ref{CLcocycle}.
	\begin{defn} [Automorphism] \label{Aut:def} Let $X$ be a $G$-system. A measure-preserving transformation $u:X\rightarrow X$ is called an \textit{automorphism} if the induced action on $L^2(X)$ by $V_u(f)=f\circ u$ commutes with the action of $G$.
	\end{defn}
	The following result is due to Bergelson Tao and Ziegler \cite[Lemma 5.3]{Berg& tao & ziegler}.
	\begin{lem} [Differentiation by an automorphism decreases the type] \label{diff}
		Let $k,m\geq 1$, let $G$ be a countable abelian group, let $X$ be an ergodic $G$-system, and let $\rho:G\times X\rightarrow S^1$ be a cocycle of type $<m$. Then, for every automorphism $t:X\rightarrow X$ which preserves $Z_{<k}(X)$, the cocycle $\Delta_t \rho(g,x)$ is of type $<m-\min(m,k)$.
	\end{lem}
	We note that Bergelson Tao and Ziegler prove the lemma above only for automorphisms of specific form, but the same proof shows that the claim holds in this generality.\\
	
	In a similar manner we have the following version for phase polynomials.
	\begin{lem} \label{diffpoly}
		Let $k,m\geq 1$, let $G$ be a countable abelian group, and let $X$ be an ergodic $G$ system. If $f:X\rightarrow S^1$ is a phase polynomial of degree $<m$, then $\Delta_t f(x)$ is of degree $<m-\min(m,k)$.
	\end{lem}
	\begin{proof}
	This lemma is proved in the proof of	\cite[Lemma 8.8]{Berg& tao & ziegler}.
	\end{proof}
	The following characterization of phase polynomials of degree $<k$ is due to Bergelson Tao and Ziegler \cite[Lemma 4.3 (iii)]{Berg& tao & ziegler}.
	\begin{lem} \label{polytype}
		Let $G$ be a countable abelian group and $X$ be an ergodic $G$-system. Then a function $f:X\rightarrow S^1$ is a phase polynomial of degree $<k$ if and only if $d^{[k]}f(x) = 1$ for $\mu^{[k]}$-almost every $x\in X^{[k]}$.
	\end{lem}
	It is natural to ask whether a cocycle of type $<k$ is cohomologous to a phase polynomial of degree $<k$. This is true for $\mathbb{F}_p^\omega$-systems \cite{Berg& tao & ziegler} (at least if $p>k$), but wrong for general groups (see e.g. \cite{HK2} or \cite[Section 9]{OS}). However in the case $k=1$ we have the following result by Moore and Schmidt \cite{MS} and Furstenberg and Weiss \cite[Lemma 10.3]{F&W}.
	\begin{thm}[Cocycles of type $<1$ are cohomologous to constants] \label{type0}
		Let $G$ be a countable abelian group. Let $X$ be an ergodic $G$-system and $\rho:G\times X\rightarrow S^1$ be a cocycle of type $<1$. Then, there exists a character $c:G\rightarrow S^1$ and a measurable map $F:X\rightarrow S^1$ such that $\rho(g,x)=c(g)\cdot\Delta_g F(x)$, for every $g\in G$ and $\mu$-almost every $x\in X$.
	\end{thm}
	\begin{prop} \label{everything}
		Let $G$ be a countable abelian group. Let $m,k\geq 1$, and suppose that $X$ is an ergodic $G$-system of order $<k+1$ and $P:X\rightarrow S^1$ a phase polynomial of degree $<m$. Then the following holds. \begin{itemize}
			\item {There exists a compact abelian group $U$ and a cocycle $\rho:G\times X\rightarrow U$ such that $X=Z_{<k}(X)\times_{\rho} U$. Moreover, if $k=2$, then for every $\chi\in\hat U$, $\chi\circ\rho$ is a C.L.\ cocycle with respect to $Z=Z_{<2}(X)$.}
			\item{Let $X=Z_{<k}(X)\times_\rho U$. Then for every $u\in U$, $\Delta_u P$ is a phase polynomial of degree $<\max\{0,m-k\}$. In particular, $P$ is measurable with respect to $Z_{<m}(X)$.}
			\item{If $p:G\times X\rightarrow U$ is a phase polynomial cocycle of degree $<k$, then $X\times_p U$ is a system of order $<k$.}
		\end{itemize}
	\end{prop}
	\begin{proof}
		The first claim follows by Theorem \ref{HKextension}. If $k=2$, then $\rho$ is of type $<2$. Therefore, by Lemma \ref{diff}, $\Delta_s \rho$ is of type $<1$ for every $s\in Z$ and the  C.L.\ equation follows by Theorem \ref{type0}.\\
		Let $P:X\rightarrow S^1$ be as in the theorem. We prove by downward induction on $l$ that $P$ is measurable with respect to $Z_{<l}(X)$ for every $m\leq l \leq k+1$. The case $l=k+1$ is trivial since $Z_{<k+1}(X)=X$. Fix $m\leq l<k+1$ and assume inductively that $P$ is measurable with respect to $Z_{<l+1}(X)$, namely, there exists a polynomial $P_{l+1}:Z_{<l+1}(X)\rightarrow S^1$ such that $P=P_{l+1}\circ\pi_{l+1}$, where $\pi_{l+1}:X\rightarrow Z_{<l+1}(X)$ is the factor map. Write $Z_{<l+1}(X)=Z_{<l}(X)\times_\rho U$ for some cocycle $\rho$. By Lemma \ref{diffpoly} we have that $\Delta_u P_{l+1}$ is of degree $<m-\min\{m,l\}=0$, hence $\Delta_u P=1$. It follows that $P_{l+1}$ is invariant with respect to translations by $u\in U$. In other words, $P$ is measurable with respect to $Z_{<l}(X)$ and the case $l=m$ gives the desired result. Finally, the last claim is a direct application of Lemma \ref{polytype} and Theorem \ref{HKextension}.
	\end{proof}
	
	\section{Results about topological groups and a computation}
	\subsection{Divisible and injective groups}
	\begin{defn} \label{cocycle}
		Let $Z$ and $U$ be locally compact abelian groups. A function $k:Z\times Z\rightarrow U$ is called a cocycle if for every $r,s,t\in Z$ we have
		\begin{equation}
		\label{2coceq}
		k(rs,t)\cdot k(r,s)=k(r,st)\cdot k(s,t).
		\end{equation}
		Moreover, a cocycle is symmetric if
		\begin{equation} \label{symmetric}
		k(s,t)=k(t,s)
		\end{equation}
		for every $s,t\in Z$.
	\end{defn}
	\begin{prop} \label{split}
		Let $Z$ and $U$ be locally compact abelian groups and let $k:Z\times Z\rightarrow U$ be a symmetric cocycle. If one of the following holds
		\begin{itemize}
			\item {$U$ is a torus. Or,}
			\item {$U,Z$ are discrete and $U$ is divisible.}
		\end{itemize} Then there exists a continuous function $\varphi:Z\rightarrow U$ such that $k(s,t) = \frac{\varphi(st)}{\varphi(s)\varphi(t)}$.
	\end{prop}
	\begin{proof}
		Without loss of generality we may assume that $k(1,1)=1_U$. From equation (\ref{2coceq}) we see that $k(1,t)=k(t,1)=1_U$ for all $t\in Z$. The cocycle $k$ induces a multiplication on the set $K=Z\times U$ by
		$(s,u)\cdot (t,v)=(st,k(s,t)uv)$. Equations (\ref{2coceq}) and (\ref{symmetric}) imply that $K$ is an abelian group. Observe, that we have a short exact sequence $$1\rightarrow U\overset{\iota}{\rightarrow} K\overset{p}{\rightarrow} Z\rightarrow 1$$ where $\iota(u)=(1,u)$ and $p(z,u)=z$. By the assumptions in the claim the short exact sequence splits. Therefore, there exists an homomorphism $q:Z\rightarrow K$ with $p(q(z))=z$. Let $\varphi:Z\rightarrow U$ be such that $q(z) = (z,\varphi(z))$. Since $q$ is a homomorphism, the claim follows.
	\end{proof}
	\subsection{Polish spaces and group actions}
	Polish groups and polish spaces (homogeneous spaces in particular) play an important role in this paper.\\ Below we summarize some important results.\\
	
	We start with the definition of a Borel cross section.
	\begin{defn} \label{cross section}
	Let $K$ be a quotient of a topological group $G$ and let $q:G\rightarrow K$ be the quotient map. A Borel cross section for $q$ is a Borel measurable map $s:K\rightarrow G$ satisfying that $q\circ s:K\rightarrow K$ is the identity map.
	\end{defn}
	\begin{thm} [The open mapping Theorem]  \label{openmap} \cite[Chapter 1]{BK}
		Let $\mG$ and $\mathcal{H}$ be Polish groups and let $p:\mG\rightarrow \mathcal{H}$ be a surjective continuous homomorphism. Then $p$ is open and there exists a Borel cross section $s:\mathcal{H}\rightarrow \mG$ such that $p\circ s = Id$.
	\end{thm}
	This theorem leads to the following results about quotient spaces.
	\begin{thm} \label{locallycomp}
	    Let $\mG$ be a polish group and let $\mH$ be a closed normal subgroup of $\mG$. Then $\mG$ is locally compact (resp. compact) if and only if $\mH$ and $\mG/\mH$ are locally compact (resp. compact).
	\end{thm}
	\begin{thm}  \cite{Effros} \label{quotient}
		If $\mG$ is a locally compact polish group which acts transitively on a compact metric space $X$. Then for any $x\in X$ the stabilizer $\Gamma = \{g\in \mG : gx=x\}$ is a closed subgroup of $\mG$ and $X$ is homeomorphic to $\mG/\Gamma$.
	\end{thm}
	\subsection{A computation}
	We will need the following computation for the Khintchine recurrence.
	\begin{prop} \label{computation}
		Let $a,b\in\mathbb{Z}$ be coprime and $U$ be a compact abelian group. Suppose that $U$ is $a,b,a+b$ and $b-a$ divisible. Then the sets $$A=\{(g,gg_1^ag_2^{\binom{a}{2}},gg_1^bg_2^{\binom{b}{2}},gg_1^{a+b}g_2^{\binom{a+b}{2}}):g,g_1,g_2\in U\}$$
		and 
		$$B = \{(u^{a+b},t\cdot u^{b-a},tv^{b-a},v^{a+b})\in U^4 : u,t,v\in U\}$$
		are equal.
	\end{prop}
	\begin{proof}
		We first prove that $A\subseteq B$. To see this fix any $g,g_1,g_2\in U$. Let $s\in U$ be such that $s^2=g_2$. Choose $u\in U$ such that $u^{a+b}=g$ and set $v=ug_1s^{a+b-1}$ and $t=gg_1^ag_2^{\binom{a}{2}}\cdot u^{a-b}$. Clearly, $v^{a+b}=gg_1^{a+b}g_2^{\binom{a+b}{2}}$ and it left to show that 
		$$gg_1^ag_2^{\binom{a}{2}}\cdot u^{a-b}\cdot v^{b-a} = gg_1^b g_2^{\binom{b}{2}}.$$
		We substitute $v=ug_1s^{a+b-1}$ above and get
		$$ g_1^{b-a}s^{(a+b-1)(b-a)} = g_1^{b-a}g_2^{\binom{b}{2}-\binom{a}{2}}.$$
		Since either $(a+b-1)$ or $b-a$ is even, we get that the equality holds.\\
		
		As for the second inclusion fix any $u,t,v\in U$. Set $g=u^{a+b}$ and for every $g_2\in U$ choose $s=s(g_2)$ such that $s^2=g_2$ and set $g_1=vu^{-1}s^{1-a-b}$. It is left to find $g_2$ such that the following equations hold
		\[
		\begin{cases}
		u^{a+b}\cdot (vu^{-1}s^{1-a-b})^a s^{a^2-a}=t \cdot u^{b-a}\\
		u^{a+b}\cdot (vu^{-1}s^{1-a-b})^b s^{b^2-b} = t\cdot v^{b-a}
		\end{cases}
		\]
		Rearranging the equations we get,
		\[
		s^{ab} = t^{-1}\cdot(uv)^{a}.
		\]
		Since $U$ is $a$ and $b$ divisible, there is a solution for $s$ and we can take $g_2=s^2$.
	\end{proof}

	\address{Einstein Institute of Mathematics\\
		The Hebrew University of Jerusalem\\
		Edmond J. Safra Campus, Jerusalem, 91904, Israel \\ Or.Shalom@mail.huji.ac.il}

\begin{thebibliography}{9}
		\bibitem{ABB}
		E. Ackelsberg, V. Bergelson, A. Best, \emph{Multiple recurrence and large intersections for abelian group actions} available at arXiv:2101.02811v2.
		\bibitem{Green}
		L. Auslander, L. Green, F. Hahn, \emph{Flows on homogeneous spaces}. Annals of Mathematics Studies, No. 53 (1963).
		\bibitem{BK}
		H. Becker, A. S. Kechris, \emph{The Descriptive Theory of Polish Groups Actions},
		London Math. Soc. Lecture Notes Ser. 232, Cambridge Univ. Press, Cambridge
		(1996).
		\bibitem{Be}
		V. Bergelson, \emph{Weakly mixing PET} Ergodic Theory Dynam. Systems 7, 337-349 (1987).
		\bibitem{BFe}
		 V. Bergelson and A. Ferr\'e Moragues. \emph{An ergodic correspondence principle, invariant
means and applications}. Israel J. Math., to appear. arXiv:2003:03029.
		\bibitem{BHK}
		V. Bergelson, B. Host and B. Kra, \emph{Multiple recurrence and nilsequences, (with an appendix by I.
			Ruzsa)}, Invent. Math. 160, no. 2, 261-303 (2005).
		\bibitem{Berg& tao & ziegler}
		V. Bergelson, T. Tao and T. Ziegler, \emph{An Inverse Theorem for The Uniformity Seminorms Associated with The Action of $\mathbb{F}_p^\infty$}, Geom. Funct. Anal. 19, No. 6, 1539-1596 (2010).
		\bibitem{BTZ}
		V. Bergelson, T. Tao, and T. Ziegler. \emph{Multiple recurrence and convergence results associated
			to $F_p^\omega$-actions}. Journal d'Analyse Mathematiqu{\'e}, 127:329–378, (2015).
		\bibitem{CL84}
		J.P. Conze and E. Lesigne, \emph{Th{\'e}or$\grave{e}$mes ergodiques pour des mesures diagonales},
		Bull. Soc. Math. France 112, 143–175 (1984).
		\bibitem{CL87}
		J.P. Conze and E. Lesigne,\emph{Sur un th{\'e}or$\grave{e}$me ergodique pour des mesures diagonales,
			in Probabilit{\'e}s}, 1–31, Publ. Inst. Rech. Math. Rennes 1987-1. (1987)
		\bibitem{CL88}
		J.P. Conze and E. Lesigne, \emph{Sur un th\'eor$\grave{e}$me ergodique pour des mesures diagonales},
		C. R. Acad. Sci. Paris, S{\'e}r. I , 306, 491–493 (1988).
				\bibitem{Effros}
		E. G. Effros, \emph{Transformation groups and $C^\star$-algebras}, Annals of Mathematics 81, 38–55 (1965).
		\bibitem{EW}
		M. Einsiedler, T. Ward, \emph{Ergodic theory with a view towards number theory.}
Graduate Texts in Mathematics, 259. Springer-Verlag London, Ltd., London. xviii+481 pp. (2011)
		\bibitem{Folner}
		E. F{\o}lner, \emph{On groups with full Banach mean value}. Math. Scand., 3, 243-254 (1955).
		\bibitem{Fran}
		N. Frantzikinakis, \emph{Multiple ergodic averages for three polynomials and applications}, Trans. Amer.
		Math. Soc. 360, 5435-5475 (2008).
		\bibitem{F1}
		H. Furstenberg,
		\emph{Ergodic behavior of diagonal measures and a theorem of Szemer{\'e}di on arithmetic progression},
		Princeton university press,
		(1981).
		\bibitem{FurRoth}
		H. Furstenberg, \emph{Recurrence in ergodic theory and combinatorial number theory.} M. B. Porter Lectures. Princeton University Press, Princeton, N.J., (1981).
		\bibitem{F&W}
		H. Furstenberg and B. Weiss, \emph{A Mean ergodic theorem for $\frac{1}{N}\sum_{n=1}^{N}f(T^n(x))g(T^{n^2}(x))$ , in Convergence in Ergodic Theory and Probability}, (Columbus, OH 1993) (Bergelson, March, and Rosenblatt, eds.), Ohio state Univ. Math. Res. Inst. Publ. 5, de Gruyte, Berlin, 193-227 (1996).
		\bibitem{G}
		T. Gowers, \emph{A new proof of Szemeredi’s theorem}, Geom. Func. Anal., 11, 465-588 (2001).
		\bibitem{HM}
		K.H. Hofmann and S.A. Morris, 
		\emph{The structure of compact groups. 
		A primer for the student—a handbook for the expert. Third edition, revised and augmented}. De Gruyter Studies in Mathematics, 25. De Gruyter, Berlin, (2013).
		\bibitem{HK2}
		B. Host and B. Kra, \emph{
			An odd Furstenberg-Szemer{\'e}di theorem and quasi-affine systems}.
		J. Anal. Math. 86, 183–220  (2002).
		\bibitem{HK}
		B. Host and B. Kra,\emph{ Nonconventional ergodic averages and nilmanifolds}, Ann. of Math. (2) 161, no. 1, 397-488 (2005).
		\bibitem{Leib}
	    A. Leibman, \emph{Polynomial sequences in groups.} J. Algebra 201, no. 1, 189–206 (1998). 
		\bibitem{L}
		A. Leibman, \emph{Host-Kra and Ziegler factors and convergence of multiple averages}, Handbook of Dynamical Systems, vol. 1B, B. Hasselblatt and A. Katok, eds., Elsevier, 841-85 (2005).
		\bibitem{Leib2}
		A. Leibman, Pointwise convergence of ergodic averages for polynomial sequences of translations on a nilmanifold, Ergodic Theory Dynam. Systems 25, 201–213 (2005).
			\bibitem{MS}
		C. Moore, K. Schmidt, \emph{Coboundaries and homomorphisms for non-singular actions and a problem
			of H. Helson} Proc. London Math. Soc. (3), 40, 443–475 (1980).
		\bibitem{M}
		S.A. Morris, \emph{Pontryagin duality and the structure of locally compact abelian groups}, London Math. Soc. Lecture Note Series, 29. Cambridge University Press, (1977).
		\bibitem{Parryb}
		W. Parry,  \emph{Ergodic properties of affine transformations and flows
			on nilmanifolds}. Amer. J. Math. 91 :757-771 (1969).
		\bibitem{Parryc}
		W. Parry, \emph{Dynamical systems on nilmanifolds}. Bull. London Math. Soc.
		2: 37-40 (1970).
		\bibitem{Rud}
		D. J. Rudolph, \emph{Eigenfunctions of $T \times S$ and the Conze-Lesigne algebra., Ergodic theory and its connections
with harmonic analysis}. Proceedings of the 1993 Alexandria conference, Alexandria, Egypt, pp. 369–432 (1993).
		\bibitem{OS}
		O. Shalom, \emph{Host-Kra theory for $\bigoplus_{p\in P}\mathbb{F}_p$-systems and multiple recurrence.} Available at 	arXiv:2101.04613.
		\bibitem{OS2}
		O. Shalom, \emph{Host-Kra factors for $\bigoplus_{p\in P}\mathbb{Z}/p\mathbb{Z}$ actions and finite dimensional nilpotent systems.} Available at arXiv:2105.00446.
		\bibitem{Sz}
		E. Szemer{\'e}di, \emph{On sets of integers containing k elements in arithmetic progression}, Acta Arith. \textbf{27}, 199-245  (1975).
		\bibitem{Walsh} N.M. Walsh, \emph{Norm convergence of nilpotent ergodic averages}. Ann. of Math. (2) 175, no. 3 1667-1688 (2012).
		\bibitem{Z1}
		T. Ziegler, A non-conventional ergodic theorem for a nilsystem.
		Ergodic Theory and Dynamical Systems 25 no. 4 1357-1370 (2005).
		\bibitem{Z}
		T. Ziegler, \emph{Universal characteristic factors and Furstenberg averages}, J. Amer. Math. Soc. 20, 53–97 (2007).
		\bibitem{Zim}
		R. Zimmer, \emph{Extensions of ergodic group actions}, Illinois J. Math. 20, 373-409 (1976).
	\end{thebibliography}
\end{document}